\documentclass{article}

\usepackage{graphicx}

\usepackage{cite}
\usepackage{amsmath, amssymb, amsthm, enumerate, mathrsfs}

\usepackage[ruled,vlined,linesnumbered]{algorithm2e}
\usepackage{color}
\usepackage{float}
\usepackage{pgfplots}
\pgfplotsset{compat=1.18}

\newtheorem{theorem}{Theorem}[section]
\newtheorem{corollary}[theorem]{Corollary}
\newtheorem{lemma}[theorem]{Lemma}
\newtheorem{proposition}[theorem]{Proposition}

\theoremstyle{definition}

\newtheorem{definition}[theorem]{Definition}
\newtheorem{example}[theorem]{Example}
\newtheorem{assumption}{Assumption}[section]

\numberwithin{equation}{section}

\makeatletter
\let\c@algocf\c@theorem

\makeatother

\SetAlCapFnt{\footnotesize}
\SetAlCapNameFnt{\footnotesize}

\usepackage{caption}
\begin{document}

\makeatletter

\begin{center}
\large{\textbf{Further analysis and extension of the higher-order Newton method of Ahmadi, Chaudhry, and Zhang}}
\end{center}

\vspace{5mm}

\begin{center}
\textsc{Lucas ter Voert, Etienne de Klerk}
\end{center}

\vspace{2mm}

\footnotesize{

\noindent
\begin{minipage}{14cm}
\textbf{Abstract:} We extend a $d$\textsuperscript{th}-order Newton method for unconstrained optimization by Ahmadi, Chaudhry, and Zhang [\textit{Advances in Mathematics}, 452:109808] to optimization with SOS-convex polynomial constraints. Consider the problem of minimizing a smooth function $f:\mathbb{R}^{n}\to\mathbb{R}$ subject to SOS-convex polynomial constraints. Given an iterate $x\in\mathbb{R}^n$, Ahmadi \textit{et al}.\ define the next iterate $x^{+}$ as the minimizer of the $d$\textsuperscript{th}-order Taylor expansion of $f$ at $x$ with a regularization term of degree $d^{\prime}$, where $d^{\prime}$ is the smallest even number greater than $d$, chosen such that this polynomial is SOS-convex, subject to the constraints. Constructing this polynomial and minimizing it subject to the constraints can both be reduced in time polynomial in $n$ to a semidefinite program (SDP). We prove that, if $f$ is strongly convex and the tensor of the $d$\textsuperscript{th}-order partial derivatives of $f$ is Lipschitz continuous, then our method converges locally to the optimal solution $x^{\ast}$ with order $d$. We further prove that, under certain constraint qualifications, the set of active constraints at $x^{\ast}$ is identified locally in a single iteration. Next, we study the worst-case performance of the third-order Newton method in the unconstrained setting for two classes of univariate $f$ using performance estimation. Finally, we extend a globally convergent modification of the $d$\textsuperscript{th}-order Newton method to the setting of SOS-convex polynomial constraints.
\end{minipage}
\\

\vspace{5mm}


\hbox to14cm{\hrulefill}\par

\section{Introduction}

Newton's method is a classical algorithm for minimizing a smooth function $f:\mathbb{R}^{n}\to\mathbb{R}$. Given an iterate $x\in\mathbb{R}^{n}$, the next iterate $x^{+}$ is defined as the minimizer of the second-order Taylor expansion of $f$ at $x$, and can be computed in time polynomial in $n$ by solving a system of linear equations. It is well known that, if $f$ is strongly convex and the Hessian of $f$ is Lipschitz continuous, then Newton's method converges locally quadratically to the optimal solution $x^{\ast}$.

Since a higher-order Taylor expansion typically provides a more accurate local approximation of $f$ on a neighborhood of $x$, it is natural to ask whether the convergence properties of Newton's method can be improved by using a higher-order Taylor expansion. This approach, however, faces two obstacles. First, for every even number $d\geq4$, the $d$\textsuperscript{th}-order Taylor expansion of $f$ at $x$ may be unbounded from below, even if $f$ is strongly convex and $x$ is arbitrarily close to $x^{\ast}$. For an odd order, the situation is even more immediate, since every polynomial of odd degree is unbounded from below. Second, minimizing a polynomial of degree at most an even number $d\geq4$ is strongly $\operatorname{NP}$-hard (see, e.g., \cite{murty1987np}).

Ahmadi, Chaudhry, and Zhang \cite{ahmadi2024higher} removed these obstacles by designing a method based on the notion of an \emph{SOS-convex} polynomial (see Definition~\ref{def:sos-convex}). They define $x^{+}$ as the minimizer of the $d$\textsuperscript{th}-order Taylor expansion of $f$ at $x$ with a regularization term of degree $d^{\prime}$, where $d^{\prime}$ is the smallest even number greater than $d$, chosen such that this polynomial is SOS-convex. Constructing and minimizing this polynomial can both be reduced in time polynomial in $n$ to a semidefinite program (SDP). Under some assumptions, an SDP can be solved to arbitrary accuracy in polynomial time using interior-point methods (see, e.g., \cite{klerk2016turing}). Ahmadi, Chaudhry, and Zhang \cite{ahmadi2024higher} proved that, if $f$ is strongly convex and the tensor of $d$\textsuperscript{th}-order partial derivatives of $f$ is Lipschitz continuous, then their method converges locally to $x^{\ast}$ with order $d$.

Ahmadi, Chaudhry, and Zhang \cite{ahmadi2024higher} noted that their method can readily be extended to the setting of SOS-convex polynomial constraints. In this paper, we formally define this extension and prove that it retains the convergence properties established in the unconstrained setting. We further prove that, under certain constraint qualifications, the set of active constraints at $x^{\ast}$ is identified locally in a single iteration.

Throughout this paper, we assume that we have access to an oracle that gives the $d$\textsuperscript{th}-order partial~derivatives of $f$ at $x$, so that the coefficients of the $d$\textsuperscript{th}-order Taylor expansion of $f$ at $x$ may be obtained using a number of oracle calls that is polynomial in $n$.

\subsection{Previous research}

Many variants of Newton's method have been proposed to improve its convergence properties. For~example, Nesterov and Polyak \cite{nesterov2006cubic} define $x^{+}$ as the minimizer of the second-order Taylor expansion of $f$ at $x$ with a cubic regularization term, which can be computed to arbitrary accuracy in time polynomial in $n$ using linear algebra techniques. If $f$ is strongly convex and the Hessian of $f$ is Lipschitz continuous, then this method converges globally to $x^{\ast}$, while retaining a local quadratic convergence rate.

There has been extensive work on generalizing the idea of a regularized Newton method to a higher order. For example, Nesterov \cite{nesterov2021implementable} defines $x^{+}$ as the minimizer of the $d$\textsuperscript{th}-order Taylor expansion of $f$ at $x$ with a regularization term of degree $d+1$, chosen such that this function is convex. If $f$ is strongly convex and the tensor of $d$\textsuperscript{th}-order partial derivatives of $f$ is Lipschitz continuous, then this method converges globally to $x^{\ast}$. Doikov and Nesterov \cite{doikov2022local} further proved that it has a local convergence rate of order $d$. Nesterov \cite{nesterov2021implementable} gave an algorithm that can compute a minimizer of a cubic function with a quartic regularization term to arbitrary accuracy in polynomial time. Therefore, this method can readily be implemented for $d=3$. To the best of our knowledge, no algorithm is known that can implement this method for $d\geq4$ in a way that improves upon Newton's method.

Ahmadi and Zhang \cite{ahmadi2022complexitya} showed that finding a local minimizer of a cubic function can be reduced in linear time to an SDP. Based on this result, Silina and Zhang \cite{silina2022unregularized} define $x^{+}$ as the local minimizer of the third-order Taylor expansion of $f$ at $x$. If $f$ is strongly convex and the tensor of third-order partial derivatives of $f$ is Lipschitz continuous, then this method converges locally cubically to $x^{\ast}$. By Theorem~2.1 in \cite{ahmadi2022complexityb}, finding a local minimizer of a polynomial of degree at most $d\geq4$ is strongly $\operatorname{NP}$-hard. Therefore, this method cannot easily be generalized to a higher order, unless $\operatorname{P}=\operatorname{NP}$.

To the best of our knowledge, the method of Ahmadi, Chaudhry, and Zhang \cite{ahmadi2024higher} is the first algorithm that is known to converge locally to $x^{\ast}$ with order $d$, while requiring time polynomial in $n$ per iteration (in a sense to be made precise later). They also gave a modification of their method that converges globally to $x^{\ast}$, while retaining a local convergence rate of order $d$.

Finally, Cartis and Zhu \cite{cartis2024global} developed a variant of the method of Ahmadi, Chaudhry, and Zhang \cite{ahmadi2024higher} using adaptive regularization. This method converges globally to a stationary point of $f$, even if $f$ is nonconvex. If $f$ is strongly convex, then it retains a local convergence rate of order $d$.

\subsection{Outline and contributions}

In Section~2, we review prerequisites on higher-order Taylor expansions and SOS-convex polynomials. In Section~3, we formally define the method of Ahmadi, Chaudhry, and Zhang \cite{ahmadi2024higher} in the setting of SOS-convex polynomial constraints (Algorithm~\ref{alg:dth-order-newton-method}).

In Section~4, we strengthen a polynomial integral lower bound by Ahmadi, Chaudhry, and Zhang \cite{ahmadi2024higher}. In Section~5, we prove that the method of Ahmadi, Chaudhry, and Zhang \cite{ahmadi2024higher} in the setting of SOS-convex polynomial constraints is well-defined in the sense that, for every iterate $x\in\mathbb{R}^{n}$, the next iterate $x^{+}$ exists and is unique (Theorem~\ref{thm:well-definedness}). We then prove that it retains the convergence properties established in the unconstrained setting (Theorem~\ref{thm:convergence}). These proofs are analogous to the proofs of Ahmadi, Chaudhry, and Zhang \cite{ahmadi2024higher} in the unconstrained setting and use the polynomial integral lower bound. We further prove that, under certain constraint qualifications, the set of active constraints at $x^{\ast}$ is identified locally in a single iteration. This result follows from a theorem by Robinson \cite{robinson1974perturbed} and is the same as for sequential quadratic programming (SQP) (see, e.g., Chapter~18 in \cite{nocedal2006numerical}).

In Section~6, we study the worst-case performance of the method of Ahmadi, Chaudhry, and Zhang \cite{ahmadi2024higher} for $d=3$ in the unconstrained setting for two classes of univariate $f$ using the performance estimation as introduced by Drori and Teboulle \cite{drori2014performance}. Our work here is closely related to the work of Rubbens, Bousselmi, Hendrickx, and Glineur \cite{rubbens2025performance}, who studied the worst-case performance of second-order methods for various classes of univariate $f$.
In particular, we consider univariate quartic polynomials, as well as univariate quasi-self-concordant functions, and establish explicit local superlinear convergence rates in both cases.

In Section~7, we extend the globally convergent modification of the method of Ahmadi, Chaudhry, and Zhang \cite{ahmadi2024higher} to the setting of SOS-convex polynomial constraints (Algorithm~\ref{alg:globally-convergent-dth-order-newton-method}). We prove that it retains the convergence properties established in the unconstrained setting (Theorem~\ref{thm:global-convergence}). Finally in Section~8, we discuss some generalized settings in which our methods can be applied.

\newpage

\section{Prerequisites}

\subsection{Higher-order Taylor expansion}

Let $f:\mathbb{R}^{n}\to\mathbb{R}$ be $d$ times continuously differentiable, and let $x\in\mathbb{R}^{n}$. The tensor of $d$\textsuperscript{th}-order partial~derivatives of $f$ at $x$ is denoted by $\nabla^{d}f\left(x\right)$. Using this notation, the $d$\textsuperscript{th}-order Taylor expansion $T_{x,d}:\mathbb{R}^{n}\to\mathbb{R}$ of $f$ at $x$ is defined by $$\text{$T_{x,d}\left(y\right):=\sum_{k=0}^{d}\frac{1}{k!}\nabla^{k}f\left(x\right){\underbrace{\left[y-x,\dots,y-x\right]}_{\text{$k$ times}}}$.}$$ The remainder $R_{x,d}:\mathbb{R}^{n}\to\mathbb{R}$ of the $d$\textsuperscript{th}-order Taylor expansion of $f$ at $x$ is defined by $$\text{$R_{x,d}\left(y\right):=f\left(y\right)-T_{x,d}\left(y\right)$.}$$ The norm of a tensor $A\in\mathbb{R}^{n_{1}\times\dots\times n_{d}}$ is defined as $$\text{$\left\|A\right\|=\max_{\left\|u_{1}\right\|=\cdots=\left\|u_{d}\right\|=1}\left|A\left[u_{1},\dots,u_{d}\right]\right|$.}$$ Using this definition, the norm of a symmetric tensor $A\in\mathbb{R}^{n\times\dots\times n}$ is $$\left\|A\right\|:=\max_{\left\|u\right\|=1}\left|A\left[u,\dots,u\right]\right|,$$ see, e.g., Theorem~13.23 in \cite{guler2010foundations}.

\begin{theorem}[see, e.g., inequality (11) in \cite{baes2009estimate}]
    \label{thm:baes}
    Let $f:\mathbb{R}^{n}\to\mathbb{R}$ be $d$ times continuously differentiable. Assume that the tensor of $d$\textsuperscript{th}-order partial derivatives of $f$ is Lipschitz continuous. That is, there exists a constant $L\geq0$ such that $$\text{$\left\|\nabla^{d}f\left(y\right)-\nabla^{d}f\left(x\right)\right\|\leq L\left\|y-x\right\|$ for all $x,y\in\mathbb{R}^{n}$.}$$ Then, $$\text{$\left\|\nabla^{k}R_{x,d}\left(y\right)\right\|\leq\mathop{\frac{L}{\left(d-k+1\right)!}}\left\|y-x\right\|^{d-k+1}$ for all $x,y\in\mathbb{R}^{n}$.}$$
\end{theorem}

\subsection{SOS-convex polynomial}

Deciding whether a polynomial of degree at most an even number $d\geq4$ is nonnegative is strongly $\operatorname{NP}$-hard (see, e.g., \cite{murty1987np}). However, a tractable sufficient condition is that it is a \emph{sum of squares (SOS)}.

\begin{definition}
    A polynomial $p:\mathbb{R}^{n}\to\mathbb{R}$ is called a \emph{sum of squares (SOS)} if there exist polynomials $q_{1},\dots,q_{r}:\mathbb{R}^{n}\to\mathbb{R}$ such that $p=\sum_{i=1}^{r}q_{i}^{2}$.
\end{definition}

Note that every SOS polynomial is nonnegative. Hilbert \cite{hilbert1888darstellung} proved that every nonnegative polynomial of degree at most an even number $d\geq2$ is SOS if and only if $n=1$, $d=2$, or $\left(n,d\right)=\left(2,4\right)$ (for a modern exposition, see, e.g., \cite{reznick2000concrete}).

\begin{theorem}[see, e.g., equation (4.4) in \cite{parrilo2000structured}]
    \label{thm:parrilo}
    A polynomial $p:\mathbb{R}^{n}\to\mathbb{R}$ of degree at most an even number $d\geq2$ is SOS if and only if there exists a symmetric matrix $Q\succeq0$ such that $$\text{$p\left(x\right)=\left\langle Q\left[x\right]_{\frac{d}{2}},\left[x\right]_{\frac{d}{2}}\right\rangle$,}$$ where $\left[x\right]_{\frac{d}{2}}$ denotes the vector of monomials in $x\in\mathbb{R}^{n}$ of degree at most $\frac{d}{2}$.
\end{theorem}

The condition in Theorem~\ref{thm:parrilo} can be written as a system of linear equalities by equating the coefficients of both sides. Thus, deciding whether a polynomial of degree at most an even number $d\geq2$ is SOS can be reduced in time polynomial in $n$ to an SDP. The size of this SDP can be reduced if the polynomial is sparse (see, e.g., \cite{kojima2005sparsity}) or invariant under a group action (see, e.g., \cite{murota2010numerical}).
G\"artner, Magron, and Vallentin \cite{gartner2026sums} recently showed that the \emph{weak} membership problem for the cone of SOS polynomials of degree at most an even number $d\geq2$ can be solved in polynomial time. To the best of our knowledge, the time complexity of the strong membership problem is still unknown.

\newpage

Ahmadi, Olshevsky, Parrilo, and Tsitsiklis \cite{ahmadi2013np} proved that deciding whether a polynomial of degree at most an even number $d\geq4$ is convex is strongly $\operatorname{NP}$-hard. However, a tractable sufficient condition is that it is \emph{SOS-convex}.
    
\begin{definition}
    \label{def:sos-convex}
        A polynomial $p:\mathbb{R}^{n}\to\mathbb{R}$ is called \emph{SOS-convex} if the polynomial $q:\mathbb{R}^{n}\times\mathbb{R}^{n}\to\mathbb{R}$ defined by $q\left(x,y\right):=\left\langle\nabla^{2}p\left(x\right)y,y\right\rangle$ is SOS.
\end{definition}

Ahmadi and Parrilo \cite{ahmadi2013complete} proved that every convex~polynomial of degree at most an even number $d\geq2$ is SOS-convex if and only if $n=1$, $d=2$, or $\left(n,d\right)=\left(2,4\right)$. Since the polynomial $q$ in Definition~\ref{def:sos-convex} is sparse, the size of the SDP in Theorem~\ref{thm:parrilo} can be reduced (see, e.g., \cite{ahmadi2024higher}).

\begin{corollary}
    \label{cor:ahmadi}
    A polynomial $p:\mathbb{R}^{n}\to\mathbb{R}$ of degree at most an even number $d\geq2$ is SOS-convex if and only if there exists a symmetric matrix $Q\succeq0$ such that $$\text{$\left\langle\nabla^{2}p\left(x\right)y,y\right\rangle=\left\langle Q\left(\left[x\right]_{\frac{d}{2}-1}\otimes y\right),\left[x\right]_{\frac{d}{2}-1}\otimes y\right\rangle$,}$$  where $\left[x\right]_{\frac{d}{2}-1}$ denotes the vector of monomials in $x\in\mathbb{R}^{n}$ of degree at most $\frac{d}{2}-1$.
\end{corollary}

A fundamental result by Lasserre \cite{lasserre2009convexity} states that an SOS-convex polynomial optimization problem of degree at most an even number $d\geq2$ can be reduced in time polynomial in $n$ to another SDP.

\begin{theorem}[see Theorem~3.3 in \cite{lasserre2009convexity}]
    \label{thm:lasserre}
    An optimization problem of the form
    \begin{equation}
        \label{eq:sos-convex-polynomial-optimization-problem}
        \begin{aligned}
            \min_{x\in\mathbb{R}^{n}} &&& p\left(x\right) && \\
            \operatorname{s.t.} &&& \text{$q_{i}\left(x\right)\leq0$,} && \text{$i\in\left\{1,\dots,m\right\}$,}
        \end{aligned}
    \end{equation}
    where $p,q_{1},\dots,q_{m}:\mathbb{R}^{n}\to\mathbb{R}$ are SOS-convex polynomials of degree at most an even number $d\geq2$ can be reduced in time polynomial in $n$ to an SDP in the sense that
    \begin{itemize}
        \item if problem~\eqref{eq:sos-convex-polynomial-optimization-problem} has an optimal solution, then the SDP also has an optimal solution,
        \item given an optimal solution to the SDP, an optimal solution to problem~\eqref{eq:sos-convex-polynomial-optimization-problem} can be computed in time polynomial in $n$.
    \end{itemize}
\end{theorem}

For an explicit expression of the SDP in Theorem~\ref{thm:lasserre}, see problem~(3.3) in \cite{lasserre2009convexity}.

\medskip

To the best of our knowledge, the exact time complexity of an SDP is unknown. Ramana \cite{ramana1997exact} proved~that, in the Turing model, deciding whether an SDP is feasible belongs to $\operatorname{NP}$ if and only if it belongs to $\operatorname{co-NP}$. The proof relies on the construction of a strong dual of a given SDP, which has size polynomial in that of the original SDP. Consequently, the semidefinite feasibility problem cannot be $\operatorname{NP}$-complete, unless $\operatorname{NP}=\operatorname{co-NP}$. As another consequence, if one uses the real-number model of computation of Blum, Shub, and Smale \cite{blum1989theory}, the semidefinite feasibility problem belongs to $\operatorname{NP}\cap\operatorname{co-NP}$. Under some assumptions, an SDP can be solved to arbitrary accuracy in polynomial time using interior-point methods (see, e.g., \cite{klerk2016turing}).

\section{The method of Ahmadi, Chaudhry, and Zhang  in the setting of SOS-convex polynomial constraints}

In this paper, we consider a constrained optimization problem of the form
\begin{equation}
    \label{eq:main-problem}
    \begin{aligned}
        \min_{x\in\mathbb{R}^{n}} &&& f\left(x\right) && \\
        \operatorname{s.t.} &&& \text{$g_{i}\left(x\right)\leq0$,} && \text{$i\in\left\{1,\dots,m\right\}$,}
    \end{aligned}
\end{equation}
where $f:\mathbb{R}^{n}\to\mathbb{R}$ is $d$ times continuously differentiable, and $g_{1},\dots,g_{m}:\mathbb{R}^{n}\to\mathbb{R}$ are SOS-convex polynomials. We denote the feasible set by $$\text{$C:=\left\{x\in\mathbb{R}^{n}:g_{i}\left(x\right)\leq0,\,i\in\left\{1,\dots,m\right\}\right\}$.}$$ We assume that Slater's condition is satisfied, i.e., $\operatorname{int}\left(C\right)\neq\emptyset$, which ensures that the SDP reformulation of problem~\eqref{eq:subproblem} below satisfies strong duality.
We assume that $f$ is strongly convex, and that there exists a (unique) minimizer $x^\ast \in C$.
Given an iterate $x\in\mathbb{R}^{n}$, we first compute the regularization parameter proposed by Ahmadi, Chaudhry, and Zhang \cite{ahmadi2024higher}:
\begin{equation}
    \label{eq:t}
    \begin{aligned}
        t\left(x\right):=\min_{t\in\mathbb{R}} &&& t\\
        \operatorname{s.t.} &&& \text{$y\mapsto T_{x,d}\left(y\right)+t\left\|y-x\right\|^{d^{\prime}}$ is SOS-convex,}\\
        &&& \text{$t\geq0$,}
    \end{aligned}
\end{equation}
where $d^{\prime}$ is the smallest even number greater than $d$. Using Corollary~\ref{cor:ahmadi}, problem~\eqref{eq:t} can be reduced in time polynomial in $n$ to an SDP. Ahmadi, Chaudhry, and Zhang \cite{ahmadi2024higher} define their regularized $d$\textsuperscript{th}-order Taylor expansion $\psi_{x,d}:\mathbb{R}^{n}\to\mathbb{R}$ of $f$ at $x$ by $$\text{$\psi_{x,d}\left(y\right):=T_{x,d}\left(y\right)+t\left(x\right)\left\|y-x\right\|^{d^{\prime}}$.}$$ We now define our next iterate $x^{+}$ as the optimal solution to the problem
\begin{equation}
    \label{eq:subproblem}
    \begin{aligned}
        \min_{y\in\mathbb{R}^{n}} &&& \psi_{x,d}\left(y\right) && \\
        \operatorname{s.t.} &&& \text{$g_{i}\left(y\right)\leq0$,} && \text{$i\in\left\{1,\dots,m\right\}$.}
    \end{aligned}
\end{equation}
By Theorem~\ref{thm:lasserre}, problem~\eqref{eq:subproblem} can be reduced in time polynomial in $n$ to another SDP.

\medskip

\begin{algorithm}[H]
    \caption{$d$\textsuperscript{th}-order Newton method (cf, Algorithm~1 in \cite{ahmadi2024higher})}
    \label{alg:dth-order-newton-method}
    \KwIn{Starting point $x\in\mathbb{R}^{n}$}
    Solve problem~\eqref{eq:t} to compute $t\left(x\right)$\;
    Let $x^{+}$ be the optimal solution to problem~\eqref{eq:subproblem}\;
    Set $x\gets x^{+}$ and repeat\;
\end{algorithm}

\medskip

A few words on the computational complexity of a given iteration. First, we assume that we have access to an oracle that gives each $d$\textsuperscript{th}-order partial derivatives of $f$ at a given $x$ in unit time. This is the usual assumption in the black-box complexity setting, and it means that the coefficients of $T_{x,d}$ can be computed in time polynomial in $n$ for fixed $d$. In Step 1, $t\left(x\right)$ may be obtained to within a fixed accuracy by solving an SDP in the real-number model of computation of Blum, Shub, and Smale \cite{blum1989theory}. Subsequently, in Step 2, problem~\eqref{eq:subproblem} may likewise be solved to fixed accuracy by solving an SDP if the Slater condition holds (see, e.g., \cite{gribling2023note}). The tractability of the algorithm should be understood in this light. Note that it is not really meaningful to discuss the complexity of an iteration in the Turing model, since this would require assumptions on the bit sizes of the partial derivatives of $f$ at a given $x$.

\section{Polynomial integral lower bound}

Let $p:\mathbb{R}\to\mathbb{R}$ be a univariate polynomial of degree at most an even number $d\geq2$. Assume that $p$ is non- negative on $\left[0,1\right]$. That is, $p\left(t\right)\geq0$ for all $t\in\left[0,1\right]$. Using a quadrature rule for integration proposed in \cite{clenshaw1960method} and analyzed in \cite{imhof1963method}, Ahmadi, Chaudhry and Zhang \cite{ahmadi2024higher} proved that $$\text{$\int_{0}^{1}p\left(t\right)dt\geq\frac{1}{2\left(d^{2}-1\right)}p\left(0\right)$.}$$ It is natural to ask whether this lower bound is optimal. In this section, we determine the largest possible constant $\lambda\geq0$ such that every univariate polynomial $p:\mathbb{R}\to\mathbb{R}$ of degree at most $d$ that is nonnegative on $\left[0,1\right]$ satisfies $\int_{0}^{1}p\left(t\right)dt\geq\lambda p\left(0\right)$. Below, we show that the largest possible constant is $\lambda:=\frac{1}{\left(\frac{d}{2}+1\right)^{2}}$.

\begin{theorem}
    \label{thm:polynomial-integral-lower-bound}
    Let $p:\mathbb{R}\to\mathbb{R}$ be a univariate polynomial of degree at most an even number $d\geq2$. Assume that $p$ is nonnegative on $\left[0,1\right]$. That is, $p\left(t\right)\geq0$ for all $t\in\left[0,1\right]$. Then, $$\int_{0}^{1}p\left(t\right)dt\geq\frac{1}{\left(\frac{d}{2}+1\right)^{2}}p\left(0\right).$$ Moreover, this lower bound is optimal.
\end{theorem}

\begin{proof}
    Note that it suffices to solve the optimization problem $$
    \begin{aligned}
        \min_{p} &&& \int_{0}^{1}p\left(t\right)dt\\
        \operatorname{s.t.} &&& p\left(0\right)=1\\
         &&& \text{$p$ is nonnegative on $\left[0,1\right]$.}
    \end{aligned}
    $$ By a well-known result from approximation theory called the \emph{Markov--Luk\'acs theorem}, a univariate polynomial $p:\mathbb{R}\to\mathbb{R}$ of degree at most $d$ is nonnegative on $\left[0,1\right]$ if and only if there exist SOS polynomials $\sigma,\tau:\mathbb{R}\to\mathbb{R}$ of degree at most $d$ and $d-2$, respectively, such that $$\text{$p\left(t\right)=\sigma\left(t\right)+t\left(1-t\right)\tau\left(t\right)$.}$$ Therefore, the problem is equivalent to $$
    \begin{aligned}
        \min_{\sigma,\tau} &&& \int_{0}^{1}\left(\sigma\left(t\right)+t\left(1-t\right)\tau\left(t\right)\right)dt\\
        \operatorname{s.t.} &&& \sigma\left(0\right)=1\\
         &&& \text{$\sigma$ and $\tau$ are SOS.}
    \end{aligned}
    $$ Since $t\left(1-t\right)\geq0$ for all $t\in\mathbb{R}$, and every SOS polynomial $\tau:\mathbb{R}\to\mathbb{R}$ of degree at most $d-2$ is nonnegative, every optimal solution $\left(\sigma,\tau\right)$ must satisfy $\tau=0$. Therefore, the problem reduces to $$
    \begin{aligned}
        \min_{\sigma} &&& \int_{0}^{1}\sigma\left(t\right)dt\\
        \operatorname{s.t.} &&& \sigma\left(0\right)=1\\
         &&& \text{$\sigma$ is SOS.}
    \end{aligned}
    $$ By Theorem~\ref{thm:parrilo}, the problem is equivalent to the SDP $$
    \begin{aligned}
        \min_{Q\in\mathcal{S}^{\frac{d}{2}+1}} &&& \int_{0}^{1}\left\langle Q\left[t\right]_{\frac{d}{2}},\left[t\right]_{\frac{d}{2}}\right\rangle dt\\
        \operatorname{s.t.} &&& \left\langle Qe_{0},e_{0}\right\rangle=1\\
         &&& \text{$Q\succeq0$.}
    \end{aligned}
    $$ Note that the SDP is strictly feasible. Therefore, it satisfies strong duality. The dual SDP is $$
    \begin{aligned}
        \max_{\lambda\in\mathbb{R}} &&& \lambda\\
        \operatorname{s.t.} &&& \text{$\lambda e_{0}e_{0}^{\top}\preceq\int_{0}^{1}\left[t\right]_{\frac{d}{2}}\left[t\right]_{\frac{d}{2}}^{\top}dt$.}
    \end{aligned}
    $$ It can be shown that the optimal solution is $$\text{$\lambda:=\frac{1}{\left\langle\left(\int_{0}^{1}\left[t\right]_{\frac{d}{2}}\left[t\right]_{\frac{d}{2}}^{\top}dt\right)^{-1}e_{0},e_{0}\right\rangle}$.}$$ Choi \cite{choi1983tricks} proved that $$\text{$\left(\int_{0}^{1}\left[t\right]_{\frac{d}{2}}\left[t\right]_{\frac{d}{2}}^{\top}dt\right)^{-1}=\left(\left(-1\right)^{k+l}\left(k+l+1\right)\binom{\frac{d}{2}+k+1}{\frac{d}{2}-l}\binom{\frac{d}{2}+l+1}{\frac{d}{2}-k}\binom{k+l}{k}^{2}\right)_{k,l=0}^{\frac{d}{2}}$.}$$ Therefore, $$\text{$\lambda=\frac{1}{\left\langle\left(\int_{0}^{1}\left[t\right]_{\frac{d}{2}}\left[t\right]_{\frac{d}{2}}^{\top}dt\right)^{-1}e_{0},e_{0}\right\rangle}=\frac{1}{\left(\frac{d}{2}+1\right)^{2}}$.}$$
\end{proof}

\section{Proof of local convergence with order $d$ in the setting of SOS-convex polynomial constraints}

Throughout this section, we assume that $f$ satisfies the following regularity conditions.

\begin{assumption}
    \label{ass:strongly-convex}
    $f$ is strongly convex. That is, there exists a constant $\mu>0$ such that $$\text{$\nabla^{2}f\left(x\right)\succeq\mu I$ for all $x\in\mathbb{R}^{n}$.}$$
\end{assumption}

\begin{assumption}
    \label{ass:lipschitz}
    The tensor of $d$\textsuperscript{th}-order partial derivatives of $f$ is Lipschitz continuous. That is, there exists a constant $L\geq0$ such that $$\text{$\left\|\nabla^{d}f\left(y\right)-\nabla^{d}f\left(x\right)\right\|\leq L\left\|y-x\right\|$ for all $x,y\in\mathbb{R}^{n}$.}$$
\end{assumption}

\newpage

Ahmadi, Chaudhry, and Zhang \cite{ahmadi2024higher} proved that problem~\eqref{eq:t} is always feasible. Using analogous arguments, we can prove that problem~\eqref{eq:t} it is always \emph{strictly} feasible.

\begin{lemma}[cf., Lemma~3 in \cite{ahmadi2024higher}]
    Problem~\eqref{eq:t} is strictly feasible.
\end{lemma}

\begin{proof}
    Let $C_{x,d}$ be the sum of terms of $T_{x,d}$ of degree at least $3$. Note that $$\text{$\frac{1}{\alpha^{2}}C_{x,d}\left(x+\alpha u\right)\to0$ as $\alpha\downarrow0$.}$$ By Lemma~2 in \cite{ahmadi2024higher}, the polynomial $u\mapsto\left\|u\right\|^{2}+\left\|u\right\|^{d^{\prime}}$ lies in the interior of the cone of SOS-convex poly- nomials of degree at most $d^{\prime}$. Therefore, there exists a scalar $\alpha>0$ such that the polynomial $$u\mapsto\frac{1}{\alpha^{2}}C_{x,d}\left(x+\alpha u\right)+\frac{\mu}{2}\left(\left\|u\right\|^{2}+\left\|u\right\|^{d^{\prime}}\right)$$ also lies in the interior of this cone. Note that the quadratic function $$u\mapsto f\left(x\right)+\alpha\left\langle\nabla f\left(x\right),u\right\rangle+\frac{\alpha^{2}}{2}\left\langle\left(\nabla^{2}f\left(x\right)-\mu I\right)u,u\right\rangle$$ is convex. Therefore, it is SOS-convex. Therefore, the polynomial
    $$
    \begin{aligned}
        u &\mapsto \frac{1}{\alpha^{2}}T_{x,d}\left(x+\alpha u\right)+\frac{\mu}{2}\left\|u\right\|^{d^{\prime}}\\
         &= \frac{1}{\alpha^{2}}\left(f\left(x\right)+\alpha\left\langle\nabla f\left(x\right),u\right\rangle+\frac{\alpha^{2}}{2}\left\langle\nabla^{2}f\left(x\right)u,u\right\rangle\right)+\frac{1}{\alpha^{2}}C_{x,d}\left(x+\alpha u\right)+\frac{\mu}{2}\left\|u\right\|^{d^{\prime}}\\
         &= \frac{1}{\alpha^{2}}\left(f\left(x\right)+\alpha\left\langle\nabla f\left(x\right),u\right\rangle+\frac{\alpha^{2}}{2}\left\langle\left(\nabla^{2}f\left(x\right)-\mu I\right)u,u\right\rangle\right)+\frac{1}{\alpha^{2}}C_{x,d}\left(x+\alpha u\right)+\frac{\mu}{2}\left(\left\|u\right\|^{2}+\left\|u\right\|^{d^{\prime}}\right)
    \end{aligned}
    $$ lies in the interior of the cone of SOS-convex polynomials of degree at most $d^{\prime}$. By a scaling and an affine change of variables, the polynomial $$y\mapsto T_{x,d}\left(y\right)+\frac{\mu}{2}\frac{1}{\alpha^{d^{\prime}-2}}\left\|y-x\right\|^{d^{\prime}}$$ also lies in the interior of this cone. Hence, the value $t:=\frac{\mu}{2}\frac{1}{\alpha^{d^{\prime}-2}}$ is strictly feasible.
\end{proof}

Ahmadi, Chaudhry, and Zhang \cite{ahmadi2024higher} proved that the function $x\mapsto t\left(x\right)$ is bounded on any compact set. Since problem~\eqref{eq:t} is strictly feasible, Theorem~4.2 by Bonnans and Shapiro \cite{bonnans1998optimization} implies that the function $x\mapsto t\left(x\right)$ is continuous. Below, we give a self-contained proof without using this theorem.

\begin{lemma}
    The function $x\mapsto t\left(x\right)$ is continuous.
\end{lemma}

\begin{proof}
    Let $\varepsilon>0$. Since problem~\eqref{eq:t} is strictly feasible, the polynomials $$\text{$z\mapsto T_{x,d}\left(z\right)+\left(t\left(x\right)+\varepsilon\right)\left\|z-x\right\|^{d^{\prime}}$ and $z\mapsto T_{x,d}\left(z\right)+\left(t\left(x\right)-\varepsilon\right)\left\|z-x\right\|^{d^{\prime}}$}$$ lie in the interior and exterior, respectively, of the cone of SOS-convex polynomials of degree at most $d^{\prime}$. By continuity, there exists a $\delta>0$ such that, for every $y\in B\left(x,\delta\right)$, the polynomials $$\text{$z\mapsto T_{y,d}\left(z\right)+\left(t\left(x\right)+\varepsilon\right)\left\|z-y\right\|^{d^{\prime}}$ and $z\mapsto T_{y,d}\left(z\right)+\left(t\left(x\right)-\varepsilon\right)\left\|z-y\right\|^{d^{\prime}}$}$$ also lie in the interior and exterior, respectively, of this cone. Therefore, $$\text{$t\left(x\right)-\varepsilon<t\left(y\right)<t\left(x\right)+\varepsilon$.}$$ Hence, the function $x\mapsto t\left(x\right)$ is continuous.
\end{proof}

Ahmadi, Chaudhry, and Zhang \cite{ahmadi2024higher} proved that, in the unconstrained setting, problem~\eqref{eq:subproblem} always has a unique optimal solution. We extend their proof to our setting.

\begin{lemma}[cf., Lemma~5 in \cite{ahmadi2024higher}]
    Problem~\eqref{eq:subproblem} has a unique optimal solution.
\end{lemma}

\begin{proof}
    By the fundamental theorem of calculus and Theorem~\ref{thm:polynomial-integral-lower-bound}, $$
    \begin{aligned}
        \psi_{x,d}\left(y\right) &= \psi_{x,d}\left(x\right)+\left\langle\nabla\psi_{x,d}\left(x\right),y-x\right\rangle +\int_{0}^{1}\int_{0}^{t}\left\langle\nabla^{2}\psi_{x,d}\left(x+s\left(y-x\right)\right)\left(y-x\right),y-x\right\rangle ds\,dt\\
         &= \psi_{x,d}\left(x\right)+\left\langle\nabla\psi_{x,d}\left(x\right),y-x\right\rangle +\int_{0}^{1}t\int_{0}^{1}\left\langle\nabla^{2}\psi_{x,d}\left(x+st\left(y-x\right)\right)\left(y-x\right),y-x\right\rangle ds\,dt\\
         &\geq \psi_{x,d}\left(x\right)+\left\langle\nabla\psi_{x,d}\left(x\right),y-x\right\rangle +\int_{0}^{1}t\left(\frac{1}{\left(\frac{d^{\prime}}{2}\right)^{2}}\left\langle\nabla^{2}\psi_{x,d}\left(x\right)\left(y-x\right),y-x\right\rangle\right)dt\\
         &= \psi_{x,d}\left(x\right)+\left\langle\nabla\psi_{x,d}\left(x\right),y-x\right\rangle+\frac{1}{\left(\frac{d^{\prime}}{2}\right)^{2}}\frac{1}{2}\left\langle\nabla^{2}\psi_{x,d}\left(x\right)\left(y-x\right),y-x\right\rangle\\
         &= f\left(x\right)+\left\langle\nabla f\left(x\right),y-x\right\rangle+\frac{1}{\left(\frac{d^{\prime}}{2}\right)^{2}}\frac{1}{2}\left\langle\nabla^{2}f\left(x\right)\left(y-x\right),y-x\right\rangle\\
         &\geq \text{$f\left(x\right)+\left\langle\nabla f\left(x\right),y-x\right\rangle+\frac{1}{\left(\frac{d^{\prime}}{2}\right)^{2}}\frac{\mu}{2}\left\|y-x\right\|^{2}\to\infty$ as $\left\|y-x\right\|\to\infty$.}
    \end{aligned}
    $$ Hence, $\psi_{x,d}$ is coercive. Therefore, problem~\eqref{eq:subproblem} has at least one optimal solution.
    
    Suppose that problem~\eqref{eq:subproblem} has two different optimal solutions $\bar{x}$ and $\bar{y}$. Then, by convexity, the point $\bar{x}+t\left(\bar{y}-\bar{x}\right)$ is an optimal solution for every $t\in\left[0,1\right]$. Since $\psi_{x,d}$ is a polynomial, it therefore follows that $\psi_{x,d}\left(\bar{x}+t\left(\bar{y}-\bar{x}\right)\right)=\psi_{x,d}\left(\bar{x}\right)$ for all $t\in\mathbb{R}$, which contradicts the fact that $\psi_{x,d}$ is coercive. Therefore, the optimal solution must be unique.
\end{proof}

We state our first conclusion below.

\begin{theorem}[cf., Theorem~3 in \cite{ahmadi2024higher}]
    \label{thm:well-definedness}
    Assume that $f$ is strongly convex (Assumption~\ref{ass:strongly-convex}). Then, Algorithm~\ref{alg:dth-order-newton-method} is well defined in the sense that, for every iterate $x\in\mathbb{R}^{n}$,
    \begin{itemize}
        \item problem~\eqref{eq:t} is feasible,
        \item problem~\eqref{eq:subproblem} has a unique optimal solution.
    \end{itemize}
\end{theorem}

To prove that the method of Ahmadi, Chaudhry, and Zhang \cite{ahmadi2024higher} converges locally to the optimal solution with order $d$, we need the following lemma.

\begin{lemma}[cf., Lemma~6 in \cite{ahmadi2024higher}]
    \label{lem:hessian-psi}
    For all $x,y\in\mathbb{R}^{n}$ satisfying $\left\|y-x\right\|\leq r$, where $r:=\left(\frac{\left(d-1\right)!\mu}{2L}\right)^{\frac{1}{d-1}}$, $$\text{$\nabla^{2}\psi_{x,d}\left(y\right)\succeq\frac{\mu}{2}I$.}$$
\end{lemma}

Ahmadi, Chaudhry, and Zhang \cite{ahmadi2024higher} proved that, in the unconstrained setting, their method converges locally to the optimal solution $x^{\ast}$ with order $d$. We will now extend their proof to the constrained setting.

\begin{theorem}[cf., Theorem~4 in \cite{ahmadi2024higher}]
    \label{thm:convergence}
    Assume that $f$ is strongly convex (Assumption~\ref{ass:strongly-convex}) and that the tensor of $d$\textsuperscript{th}-order partial derivatives of $f$ is Lipschitz continuous (Assumption~\ref{ass:lipschitz}). Then, for every iterate $x\in\mathbb{R}^{n}$ satisfying $\left\|x-x^{\ast}\right\|\leq r$, where $r:=\left(\frac{\left(d-1\right)!\mu}{2L}\right)^{\frac{1}{d-1}}$, the next iterate $x^{+}$ of Algorithm~\ref{alg:dth-order-newton-method} satisfies $$\text{$\left\|x^{+}-x^{\ast}\right\|\leq\left(\frac{d^{\prime}}{2}\right)^{2}\frac{2}{\mu}\left(\frac{L}{d!}+\left(\max_{\left\|z-x^{\ast}\right\|\leq r}t\left(z\right)\right)d^{\prime}\max\left\{r,1\right\}\right)\left\|x-x^{\ast}\right\|^{d}$.}$$
\end{theorem}

\begin{proof}
    Let $x\in\mathbb{R}^{n}$ satisfying $\left\|x-x^{\ast}\right\|\leq r$. Then, by Lemma~\ref{lem:hessian-psi}, $$\text{$\nabla^{2}\psi_{x,d}\left(x^{\ast}\right)\succeq\frac{\mu}{2}I$.}$$
    Since $x^{\ast}$ minimizes $f$ over $C$, $\left\langle\nabla f\left(x^{\ast}\right),y-x^{\ast}\right\rangle\geq0$ for all $y\in C$. Taking $y:=x^{+}$ gives $$\text{$\left\langle\nabla f\left(x^{\ast}\right),x^{+}-x^{\ast}\right\rangle\geq0$.}$$ Since $x^{+}$ minimizes $\psi_{x,d}$ over $C$, $\left\langle\nabla\psi_{x,d}\left(x^{+}\right),y-x^{+}\right\rangle\geq0$ for all $y\in C$. Taking $y:=x^{\ast}$ gives $$\text{$\left\langle\nabla\psi_{x,d}\left(x^{+}\right),x^{\ast}-x^{+}\right\rangle\geq0$.}$$ Adding the two inequalities gives $$\text{$\left\langle\nabla f\left(x^{\ast}\right)-\nabla\psi_{x,d}\left(x^{+}\right),x^{+}-x^{\ast}\right\rangle\geq0$.}$$ By the fundamental theorem of calculus, $$\text{$\nabla\psi_{x,d}\left(x^{+}\right)=\nabla\psi_{x,d}\left(x^{\ast}\right)+\int_{0}^{1}\nabla^{2}\psi_{x,d}\left(x^{\ast}+t\left(x^{+}-x^{\ast}\right)\right)\left(x^{+}-x^{\ast}\right)dt$.}$$ Therefore, $$\left\langle\nabla f\left(x^{\ast}\right)-\nabla\psi_{x,d}\left(x^{\ast}\right),x^{+}-x^{\ast}\right\rangle\geq\int_{0}^{1}\left\langle\nabla^{2}\psi_{x,d}\left(x^{\ast}+t\left(x^{+}-x^{\ast}\right)\right)\left(x^{+}-x^{\ast}\right),x^{+}-x^{\ast}\right\rangle dt.$$ By Theorem~\ref{thm:polynomial-integral-lower-bound} and Lemma~\ref{lem:hessian-psi}, $$\text{$\left\langle\nabla f\left(x^{\ast}\right)-\nabla\psi_{x,d}\left(x^{\ast}\right),x^{+}-x^{\ast}\right\rangle\geq\frac{1}{\left(\frac{d^{\prime}}{2}\right)^{2}}\left\langle\nabla^{2}\psi_{x,d}\left(x^{\ast}\right)\left(x^{+}-x^{\ast}\right),x^{+}-x^{\ast}\right\rangle\geq\frac{1}{\left(\frac{d^{\prime}}{2}\right)^{2}}\frac{\mu}{2}\left\|x^{+}-x^{\ast}\right\|^{2}$.}$$ Therefore, $$\text{$\left\|x^{+}-x^{\ast}\right\|\leq\left(\frac{d^{\prime}}{2}\right)^{2}\frac{2}{\mu}\left\|\nabla f\left(x^{\ast}\right)-\nabla\psi_{x,d}\left(x^{\ast}\right)\right\|$.}$$ By Theorem~\ref{thm:baes} and Lemma~\ref{lem:hessian-psi}, $$
    \begin{aligned}
        \left\|\nabla f\left(x^{\ast}\right)-\nabla\psi_{x,d}\left(x^{\ast}\right)\right\| &= \left\|\nabla R_{x,d}\left(x^{\ast}\right)-t\left(x\right)d^{\prime}\left\|x^{\ast}-x\right\|^{d^{\prime}-2}\left(x^{\ast}-x\right)\right\|\\
         &\leq \left\|\nabla R_{x,d}\left(x^{\ast}\right)\right\|+t\left(x\right)d^{\prime}\left\|x^{\ast}-x\right\|^{d^{\prime}-1}\\
         &\leq \frac{L}{d!}\left\|x^{\ast}-x\right\|^{d}+t\left(x\right)d^{\prime}\max\left\{r,1\right\}\left\|x^{\ast}-x\right\|^{d}\\
         &= \text{$\left(\frac{L}{d!}+\left(\max_{\left\|z-x^{\ast}\right\|\leq r}t\left(z\right)\right)d^{\prime}\max\left\{r,1\right\}\right)\left\|x^{\ast}-x\right\|^{d}$.}
    \end{aligned}
    $$ Therefore, $$\text{$\left\|x^{+}-x^{\ast}\right\|\leq\left(\frac{d^{\prime}}{2}\right)^{2}\frac{2}{\mu}\left(\frac{L}{d!}+\left(\max_{\left\|z-x^{\ast}\right\|\leq r}t\left(z\right)\right)d^{\prime}\max\left\{r,1\right\}\right)\left\|x-x^{\ast}\right\|^{d}$.}$$
\end{proof}

\subsection{Active-set identification}

With reference to problem \eqref{eq:main-problem}, we denote the set of active constraints at a feasible solution $x\in C$ by $$\text{$I\left(x\right):=\left\{i\in\left\{1,\dots,m\right\}:g_{i}\left(x\right)=0\right\}$.}$$

\begin{assumption}
    \label{ass:licq}
    The gradients $\left\{\nabla g_{i}\left(x^{\ast}\right)\right\}_{i\in I\left(x^{\ast}\right)}$ of the active constraints at the optimal solution $x^{\ast}$ are linearly independent.
\end{assumption}

If the gradients of the active constraints at $x^{\ast}$ are linearly independent, then there exist Lagrange~multipliers $\lambda_{1}^{\ast},\dots,\lambda_{m}^{\ast}\geq0$ such that the Karush--Kuhn--Tucker (KKT) conditions (see, e.g., Theorem~12.1 in \cite{nocedal2006numerical}) are satisfied:
\begin{itemize}
    \item stationarity: $\nabla f\left(x^{\ast}\right)+\sum_{i=1}^{m}\lambda_{i}^{\ast}\nabla g_{i}\left(x^{\ast}\right)=0$;
    \item complementary slackness: $\lambda_{i}^{\ast}g_{i}\left(x^{\ast}\right)=0$ for every $i\in\left\{1,\dots,m\right\}$;
    \item primal feasibility: $g_{i}\left(x^{\ast}\right)\leq0$ for every $i\in\left\{1,\dots,m\right\}$;
    \item dual feasibility: $\lambda_{i}^{\ast}\geq0$ for every $i\in\left\{1,\dots,m\right\}$.
\end{itemize}
Note that complementary slackness means that $\lambda_{i}^{\ast}=0$ for every $i\notin I\left(x^{\ast}\right)$.

\begin{assumption}[Strict complementary slackness]
    \label{ass:strict-complementary-slackness}
    $\lambda_{i}^{\ast}>0$ for every $i\in I\left(x^{\ast}\right)$.
\end{assumption}

We denote the space of critical directions a feasible solution $x\in C$ by $$\text{$\Lambda\left(x\right)=\left\{u\in\mathbb{R}^{n}:\text{$\left\langle\nabla g_{i}\left(x\right),u\right\rangle=0$, $i\in I\left(x\right)$}\right\}$.}$$ Since $f$ is strongly convex (Assumption~\ref{ass:strongly-convex}) and $g_{1},\dots,g_{m}$ are convex, the following second-order sufficient condition (SOSC) (see, e.g., Theorem~12.6 in \cite{nocedal2006numerical}) is trivially satisfied: $$\text{$\left\langle\left(\nabla^{2}f\left(x^{\ast}\right)+\sum_{i=1}^{m}\lambda_{i}^{\ast}\nabla^{2}g_{i}\left(x^{\ast}\right)\right)u,u\right\rangle>0$ for all $u\in\Lambda\left(x^{\ast}\right)$.}$$

\newpage

\noindent Robinson \cite{robinson1974perturbed} proved the following theorem on how continuous perturbations of nonlinear programs affect their KKT points.

\begin{theorem}[Theorem~2.1 in \cite{robinson1974perturbed}]
    \label{thm:robinson}
    Consider a parameteric family of optimization problems $$
    \begin{aligned}
        \min_{x\in\mathbb{R}^{n}} &&& f\left(x,\theta\right) && \\
        \operatorname{s.t.} &&& \text{$g_{i}\left(x,\theta\right)\leq0$,} && \text{$i\in\left\{1,\dots,m\right\}$,}
    \end{aligned}
    $$ where $f,g_{1},\dots,g_{m}:\mathbb{R}^{n}\times\mathbb{R}^{p}\to\mathbb{R}$ are twice partially continuously differentiable with respect to $x$.
    
    Let $x^{\ast}\in\mathbb{R}^{n}$ and $\theta^{\ast}\in\mathbb{R}^{p}$. Assume that the gradients of the active constraints at $x^{\ast}$ for parameter $\theta^{\ast}$ are linearly independent, that there exist Lagrange multipliers $\lambda_{1}^{\ast},\dots,\lambda_{m}^{\ast}\geq0$ satisfying the KKT conditions for parameter $\theta^{\ast}$ with strict complementary slackness, and that the SOSC for parameter $\theta^{\ast}$ is satisfied. Then, there exist a radius $R>0$ and continuous functions $X:B\left[\theta^{\ast},R\right]\to\mathbb{R}^{n}$ and $\Lambda_{1},\dots,\Lambda_{m}:B\left[\theta^{\ast},R\right]\to\mathbb{R}_{+}$ such that
    \begin{itemize}
        \item $X\left(\theta^{\ast}\right)=x^{\ast}$ and $\Lambda_{i}\left(\theta^{\ast}\right)=\lambda_{i}^{\ast}$ for every $i\in\left\{1,\dots,m\right\}$,
        \item for every parameter $\theta\in\mathbb{R}^{p}$ satisfying $\left\|\theta-\theta^{\ast}\right\|\leq R$, the gradients of the active constraints at $X\left(\theta\right)$ for parameter $\theta$ are linearly independent, $X\left(\theta\right),\Lambda_{1}\left(\theta\right),\dots,\Lambda_{m}\left(\theta\right)$ satisfy the KKT conditions for parameter $\theta$ with strict complementary slackness, and the SOSC for parameter $\theta$ is satisfied.
    \end{itemize}
\end{theorem}

Consider problem~\eqref{eq:subproblem} with parameter $x\in\mathbb{R}^{n}$. Since the functions $x\mapsto T_{x,d}$ and $x\mapsto t\left(x\right)$ are~continuous, the function $x\mapsto\psi_{x,d}$ is also continuous. Since $\nabla\psi_{x^{\ast},d}\left(x^{\ast}\right)=\nabla f\left(x^{\ast}\right)$ and $\nabla^{2}\psi_{x^{\ast},d}\left(x^{\ast}\right)=\nabla^{2}f\left(x^{\ast}\right)$, it can be verified directly that $x^{\ast},\lambda_{1}^{\ast},\dots,\lambda_{m}^{\ast}$ satisfy the conditions of Theorem~\ref{thm:robinson} with $f\left(x,\theta\right):=\psi_{\theta,d}\left(x\right)$ and $g_{i}\left(x,\theta\right):=g_{i}\left(x\right)$ for every $i\in\left\{1,\ldots,m\right\}$, and parameter $\theta^{\ast}:=x^{\ast}$. Thus, we can establish the following lemma.

\begin{lemma}
    \label{lem:robinson-applied-to-subproblem}
    Assume that $f$ is strongly convex (Assumption~\ref{ass:strongly-convex}), that the gradients of the active constraints at $x^{\ast}$ are linearly independent (Assumption~\ref{ass:licq}), and that $x^{\ast},\lambda_{1}^{\ast},\dots,\lambda_{m}^{\ast}$ satisfy the KKT conditions with strict complementary slackness (Assumption~\ref{ass:strict-complementary-slackness}). Then, there exist a radius $R>0$ and continuous functions $X^{+}:B\left[x^{\ast},R\right]\to\mathbb{R}^{n}$ and $\Lambda_{1},\dots,\Lambda_{m}:B\left[x^{\ast},R\right]\to\mathbb{R}_{+}$ such that
    \begin{itemize}
        \item $X^{+}\left(x^{\ast}\right)=x^{\ast}$ and $\Lambda_{i}\left(x^{\ast}\right)=\lambda_{i}^{\ast}$ for every $i\in\left\{1,\dots,m\right\}$,
        \item for every iterate $x\in\mathbb{R}^{n}$ satisfying $\left\|x-x^{\ast}\right\|\leq R$, the gradients of the active constraints at $X^{+}\left(x\right)$ are linearly independent, $X^{+}\left(x\right),\Lambda_{1}\left(x\right),\dots,\Lambda_{m}\left(x\right)$ satisfy the KKT conditions for problem~\eqref{eq:subproblem} with strict complementary slackness, and the SOSC for problem~\eqref{eq:subproblem} is satisfied.
    \end{itemize}
\end{lemma}

We can now prove that the set of active constraints at $x^{\ast}$ is identified locally in a single iteration.

\begin{theorem}
    \label{thm:active-set-identification}
    Assume that the conditions of Lemma~\ref{lem:robinson-applied-to-subproblem} are satisfied. Then, there exists a radius $r>0$ such that, for every iterate $x\in\mathbb{R}^{n}$ satisfying $\left\|x-x^{\ast}\right\|\leq r$, the next iterate $x^{+}$ of Algorithm~\ref{alg:dth-order-newton-method} satisfies $$\text{$I\left(x^{+}\right)=I\left(x^{\ast}\right)$.}$$
\end{theorem}

\begin{proof}
    Let $x\in\mathbb{R}^{n}$ satisfying $\left\|x-x^{\ast}\right\|\leq R$. By Lemma~\ref{lem:robinson-applied-to-subproblem}, $X^{+}\left(x\right),\Lambda_{1}\left(x\right),\dots,\Lambda_{m}\left(x\right)$ satisfy the KKT conditions for problem~\eqref{eq:subproblem}. By convexity, $X^{+}\left(x\right)$ is the optimal solution. Hence, $$\text{$x^{+}=X^{+}\left(x\right)$.}$$ By strict complementary slackness (Assumption~\ref{ass:strict-complementary-slackness}), $$
    \begin{aligned}
         & \Lambda_{i}\left(x^{\ast}\right)>0 && \text{for all $i\in I\left(x^{\ast}\right)$,}\\
         & g_{i}\left(X^{+}\left(x^{\ast}\right)\right)<0 && \text{for all $i\notin I\left(x^{\ast}\right)$.}
    \end{aligned}
    $$ By continuity, there exists a radius $r\in\left(0,R\right]$ such that every $x\in\mathbb{R}^{n}$ satisfying $\left\|x-x^{\ast}\right\|\leq r$ satisfies $$
    \begin{aligned}
         & \Lambda_{i}\left(x\right)>0 && \text{for all $i\in I\left(x^{\ast}\right)$,}\\
         & g_{i}\left(X^{+}\left(x\right)\right)<0 && \text{for all $i\notin I\left(x^{\ast}\right)$.}
    \end{aligned}
    $$ Now, by complementary slackness, $$
    \begin{aligned}
         & g_{i}\left(X^{+}\left(x\right)\right)=0 && \text{for all $i\in I\left(x^{\ast}\right)$,}\\
         & \Lambda_{i}\left(x\right)=0 && \text{for all $i\notin I\left(x^{\ast}\right)$.}
    \end{aligned}
    $$ Hence, $I\left(X^{+}\left(x\right)\right)=I\left(x^{\ast}\right)$.
\end{proof}

\section{Performance estimation for $d=3$ in the unconstrained setting for classes of univariate $f$}

Ahmadi, Chaudhry, and Zhang \cite{ahmadi2024higher} showed that, in the unconstrained setting, the next iterate of their method for $d=3$ and a univariate function $f:\mathbb{R}\to\mathbb{R}$ is
\begin{equation}
    \label{eq:third-order-newton-univariate}
    x^{+}=
    \begin{cases}
        \displaystyle x+\frac{-2f^{\prime\prime}\left(x\right)+\sqrt[3]{8\left(f^{\prime\prime}\left(x\right)\right)^{3}-12f^{\prime}\left(x\right)f^{\prime\prime}\left(x\right)f^{\prime\prime\prime}\left(x\right)}}{f^{\prime\prime\prime}\left(x\right)} & \text{if $f^{\prime\prime\prime}\left(x\right)\neq0$}\\
        \displaystyle x-\frac{f^{\prime}\left(x\right)}{f^{\prime\prime}\left(x\right)} & \text{if $f^{\prime\prime\prime}\left(x\right)=0$.}
    \end{cases}
\end{equation}
Note that, by rewriting, the update can be characterized implicitly by the cubic equation
\begin{equation}
    \label{eq:cubic-equation}
    \text{$f^{\prime}\left(x\right)f^{\prime\prime}\left(x\right)+\left(f^{\prime\prime}\left(x\right)\right)^{2}\left(x^{+}-x\right)+\frac{f^{\prime\prime}\left(x\right)f^{\prime\prime\prime}\left(x\right)}{2}\left(x^{+}-x\right)^{2}+\frac{\left(f^{\prime\prime\prime}\left(x\right)\right)^{2}}{12}\left(x^{+}-x\right)^{3}=0$.}
\end{equation}
This implicit characterization is particularly convenient in performance estimation, as it yields a polynomial relation between the derivatives of the $f$ at $x$ and the step $x^{+}-x$.

\subsection{Quartic $f$}

Let $f:\mathbb{R}\to\mathbb{R}$ be a quartic function satisfying $x^{\ast}=0$ and $f\left(x^{\ast}\right)=0$. That is, $$\text{$f\left(x\right):=ax^{4}+bx^{3}+cx^{2}$.}$$ Note that $f$ is strongly convex with constant $\mu>0$, i.e., $$\text{$f^{\prime\prime}\left(x\right)=12ax^{2}+6bx+2c\geq\mu$ for all $x\in\mathbb{R}$,}$$ if and only if $$\text{$a\geq0$, $c\geq\frac{\mu}{2}$, and $3b^{2}\leq8a\left(c-\frac{\mu}{2}\right)$.}$$ Likewise, the third derivative of $f$ is Lipschitz continuous with constant $L>0$, i.e., $$\text{$\left|f^{\prime\prime\prime}\left(y\right)-f^{\prime\prime\prime}\left(x\right)\right|=\left|\left(24ay+6b\right)-\left(24ax+6b\right)\right|=24\left|a\right|\left|y-x\right|\leq L\left|y-x\right|$ for all $x,y\in\mathbb{R}$,}$$ if and only if $\left|a\right|\leq\frac{L}{24}$. Since $f$ is quartic, equation~\eqref{eq:cubic-equation} simplifies considerably. In particular, it reduces to  $$\text{$G\left(a,b,c,x,x^{+}\right):=\left(12ax^{2}+6bx+2c\right)\left(4a\left(x^{+}\right)^{3}+3b\left(x^{+}\right)^{2}+2cx^{+}\right)+\left(3b^{2}-8ac\right)\left(x^{+}-x\right)^{3}=0$.}$$ This allows us to formulate a \emph{performance estimation problem (PEP)}. Consider the problem of determining the largest possible distance between the next iterate $x^{+}$ and the minimizer $x^{\ast}=0$ when the starting point $x$ satisfies $\left|x\right|\leq r$, where $r>0$. The corresponding PEP is
\begin{equation}
    \label{eq:pep_quartic}
    \begin{aligned}
        \tau\left(\mu,L,r\right):=\max_{a,b,c,x,x^{+}} &&& \left|x^{+}\right|\\
        \operatorname{s.t.} &&& G\left(a,b,c,x,x^{+}\right)=0\\
         &&& \text{$0\leq a\leq\frac{L}{24}$, $c\geq\frac{\mu}{2}$, $3b^{2}\leq8a\left(c-\frac{\mu}{2}\right)$}\\
         &&& \text{$\left|x\right|\leq r$,}
    \end{aligned}
\end{equation}
where $\mu,L,r>0$ are fixed. Although this optimization problem is nonconvex, it involves only five decision variables and can therefore be solved reliably using modern global optimization software.

\newpage

\begin{example}
Let $\mu=L=r=1$. An optimal solution to problem~\eqref{eq:pep_quartic}, computed using the global~optimization solver BARON (see, e.g., \cite{zhang2025solving}), is shown in Figure~\ref{fig:quartic-example}.
\end{example}

\begin{figure}[H]
\centering
\begin{tikzpicture}
\begin{axis}[
    scale = 1.2,
    axis lines = center,
    xmin = -1.25, xmax = 1.25,
    ymin = -0.3, ymax = 0.7,
    xtick = {-1, -0.5, 0, 0.5, 1},
    ytick = {-0.2, 0, 0.2, 0.4, 0.6},
    legend pos = north west,
    legend cell align = left
]
\addplot[very thick, domain=-1.5:1.5, samples=100, color=blue]{0.042*x^4-0.074*x^3+0.549*x^2};
\addlegendentry{$f$}
\addplot[very thick, domain=-1.5:1.5, samples=100, color=red]{0.517+1.044*(x-1)+0.579*(x-1)^2+0.094*(x-1)^3+0.006*(x-1)^4};
\addlegendentry{$\psi_{x,3}$}
\node[circle, fill=black, inner sep=2pt] at (axis cs: 1.000, 0.517) {};
\node[anchor=south east] at (axis cs: 1.000, 0.517) {$\left(x,f\left(x\right)\right)$};
\node[circle, fill=black, inner sep=2pt] at (axis cs: -0.233, -0.052) {};
\node[anchor=north east] at (axis cs: -0.233, -0.052) {$\left(x^{+},\psi_{x,3}\left(x^{+}\right)\right)$};
\end{axis}
\end{tikzpicture}
\caption{An optimal solution to problem~\eqref{eq:pep_quartic} for $\mu=L=r=1$. The function $f:\mathbb{R}\to\mathbb{R}$ is given by $f\left(x\right)=ax^{4}+bx^{3}+cx^{2}$, where $a=\frac{1}{24}$, $b\approx-0.074$, and $c\approx0.549$. The starting point is $x=1$ and the next iterate is $x^{+}\approx-0.234$.}
\label{fig:quartic-example}
\end{figure}
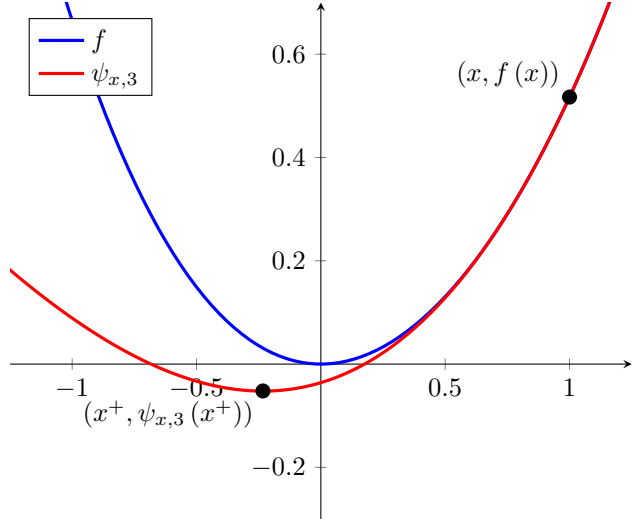

Theorem~\ref{thm:convergence} implies that, for every univariate quartic $f$ that is strongly convex with constant $\mu$ and whose third derivative is Lipschitz continuous with constant $L$, there exist a radius $r>0$ and a constant $C>0$ such that, for every iterate $x\in\mathbb{R}^{n}$ satisfying $\left|x\right|\leq r$, the next iterate $x^{+}$ satisfies $$\text{$\left|x^{+}\right|\leq C\left|x\right|^{3}$.}$$ Our objective is to determine the largest radius for which such a cubic convergence guarantee holds uniformly over all univariate quartic $f$ satisfying the above conditions.

\medskip

\noindent The following scaling properties reduce the number of fixed parameters. For every $\lambda>0$,
\begin{itemize}
    \item $a,b,c,x,x^{+}$ are feasible for $\mu,L,r$ if and only if $\lambda a,\lambda b,\lambda c,x,x^{+}$ are feasible for $\lambda\mu,\lambda L,r$, so $$\text{$\tau\left(\lambda\mu,\lambda L,r\right)=\tau\left(\mu,L,r\right)$,}$$
    \item $a,b,c,x,x^{+}$ are feasible for $\mu,L,r$ if and only if $\frac{a}{\lambda^{4}},\frac{b}{\lambda^{3}},\frac{c}{\lambda^{2}},\lambda x,\lambda x^{+}$ are feasible for $\frac{m}{\lambda^{2}},\frac{L}{\lambda^{4}},\lambda r$, so $$\text{$\tau\left(\frac{\mu}{\lambda^{2}},\frac{L}{\lambda^{4}},\lambda r\right)=\lambda\tau\left(\mu,L,r\right)$.}$$
\end{itemize}
Combining these two scaling properties gives $$\tau\left(\mu,L,r\right)=\sqrt{\frac{\mu}{L}}\tau\left(\frac{\mu^{2}}{L},\frac{\mu^{2}}{L},\sqrt{\frac{L}{\mu}}r\right)=\sqrt{\frac{\mu}{L}}\tau\left(1,1,\sqrt{\frac{L}{\mu}}r\right).$$

To study the worst-case performance for the class of univariate quartic $f$ that are strongly convex with constant $\mu$ and whose third derivatives are Lipschitz continuous with constant $L$, we fix $\mu=L=1$ and solve problem~\eqref{eq:pep_quartic} numerically for a range of $r$. For each $r\in\left\{0.1,\ldots,2.0\right\}$, we compute an upper bound on the optimal value using the global optimization solver BARON (see, e.g., \cite{zhang2025solving}). For each $r\in\left\{0.001,\ldots,2.000\right\}$, we use sequential least-squares quadratic programming (SLSQP) (see, e.g., \cite{kraft1988software}) to compute a locally optimal value. The results are shown in Figures~\ref{fig:quartic-tau} and~\ref{fig:quartic-rate}.

\begin{figure}[H]
\centering
\includegraphics[width=0.6\linewidth]{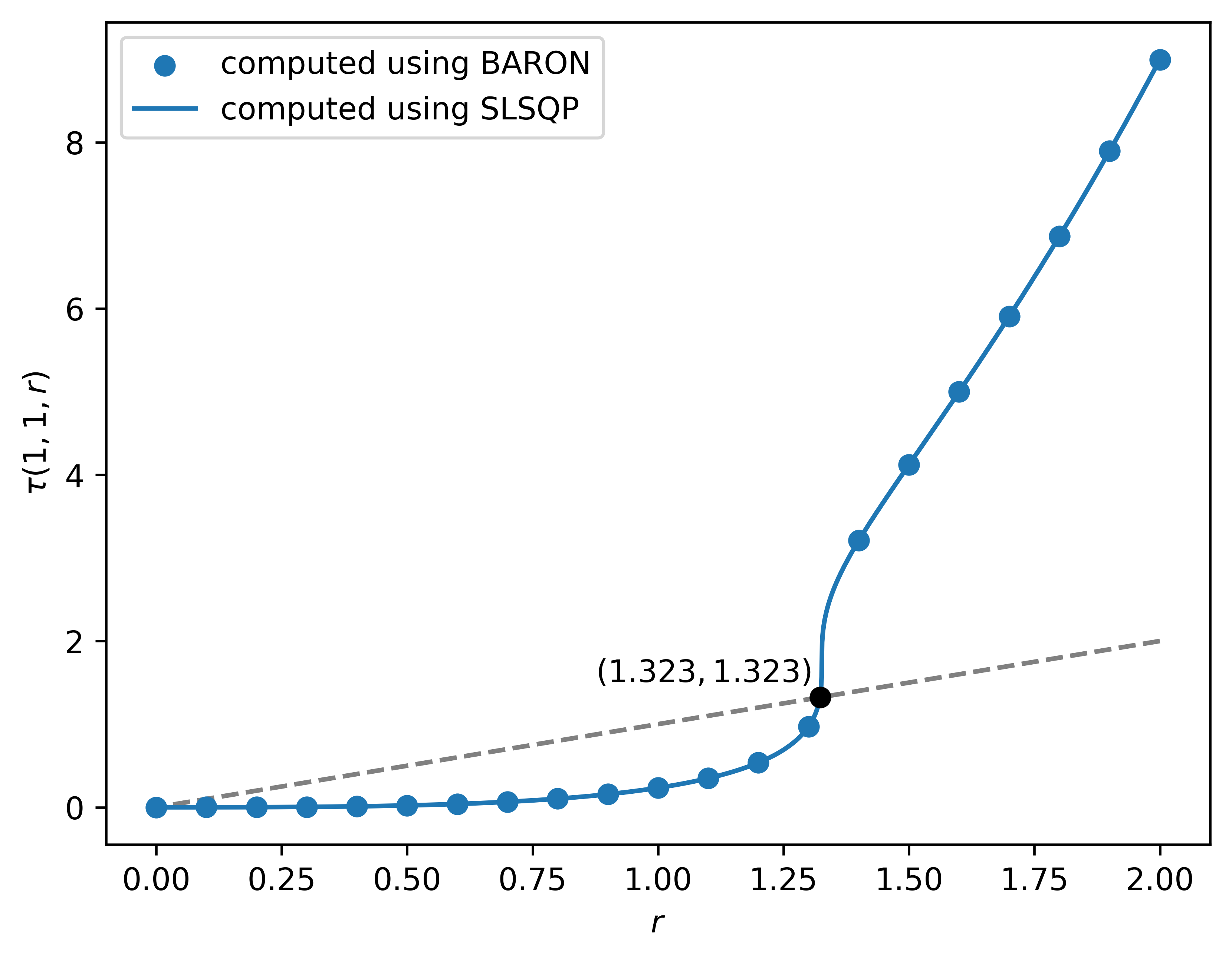}
\caption{Worst-case distance $\tau\left(1,1,r\right)$ as a function of the starting radius $r$.}
\label{fig:quartic-tau}
\end{figure}

\begin{figure}[H]
\centering
\includegraphics[width=0.6\linewidth]{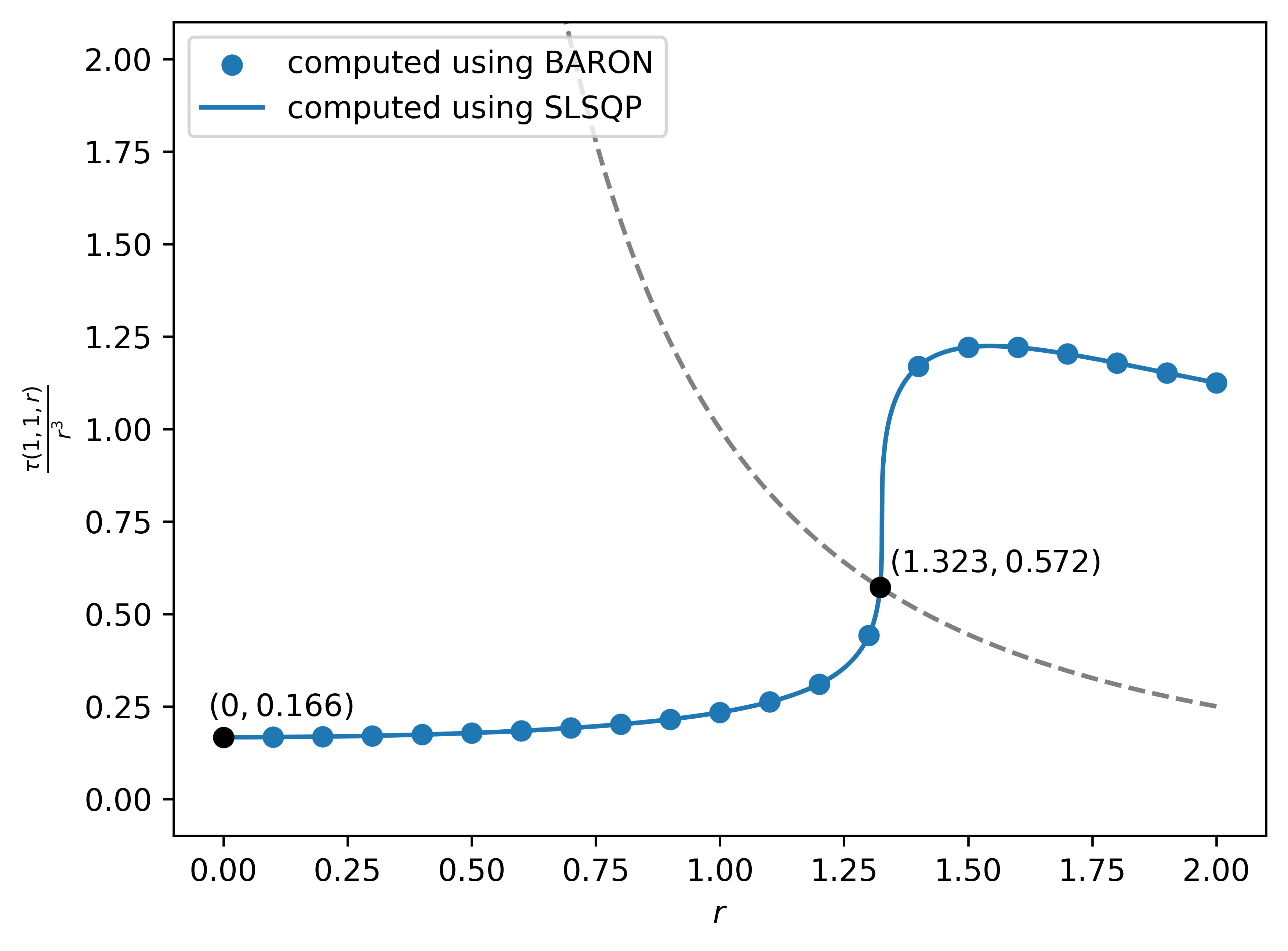}
\caption{Cubically normalized worst-case distance $\frac{\tau\left(1,1,r\right)}{r^{3}}$ as a function of the starting radius $r$.}
\label{fig:quartic-rate}
\end{figure}

For every $r>0$ for which BARON is used, the locally optimal value computed using SLSQP coincides with the upper bound on the optimal value certified by BARON. This provides strong numerical evidence that the locally optimal values computed using SLSQP are globally optimal.

Figure~\ref{fig:quartic-tau} shows that the worst-case distance $\tau\left(1,1,r\right)$ grows slowly for small values of $r$, before increasing rapidly once $r\approx1.3$. The intersection with the dashed line gives the largest starting radius for which the next iterate remains inside this radius in the worst case. Let $r^{\ast}>0$ denote the unique solution to $\tau\left(1,1,r^{\ast}\right)=r^{\ast}$, which is computed numerically to be $r^{\ast}\approx1.323$. Using the scaling properties, $$\text{$\tau\left(\mu,L,r\right)=\sqrt{\frac{\mu}{L}}\tau\left(1,1,\sqrt{\frac{L}{\mu}}r\right)<r\iff\tau\left(1,1,\sqrt{\frac{L}{\mu}}r\right)<\sqrt{\frac{L}{\mu}}r\iff\sqrt{\frac{L}{\mu}}r<r^{\ast}\iff r<\sqrt{\frac{\mu}{L}}r^{\ast}$.}$$

Figure~\ref{fig:quartic-rate} shows that $\frac{\tau\left(1,1,r\right)}{r^{3}}$ converges to a finite limit as $r\to0$, corresponding to the asymptotic cubic convergence rate. Therefore, the method converges cubically for every univariate quartic $f$ that is strongly convex with constant $\mu$ and whose third derivative is Lipschitz continuous with constant $L$ whenever $$\text{$r<\sqrt{\frac{\mu}{L}}r^{\ast}$, where $r^{\ast}\approx1.323$.}$$ The worst-case cubic convergence rate satisfies $$\text{$\frac{\tau\left(\mu,L,r\right)}{r^{3}}=\frac{L}{\mu}\frac{\tau\left(1,1,\sqrt{\frac{L}{\mu}}r\right)}{\left(\sqrt{\frac{L}{\mu}}r\right)^{3}}<\frac{L}{\mu}\frac{\tau\left(1,1,r^{\ast}\right)}{\left(r^{\ast}\right)^{3}}=\frac{L}{\mu}\frac{1}{\left(r^{\ast}\right)^{2}}$, where $\frac{1}{\left(r^{\ast}\right)^{2}}\approx0.572$.}$$

In summary, we obtain the following (tight) result characterizing the local third order convergence rate.
\begin{proposition}
    Let $f:\mathbb{R}\to\mathbb{R}$ be a univariate quartic function. Assume that $f$ is strongly convex with parameter $\mu$ (Assumption~\ref{ass:strongly-convex}) and that the third derivative of $f$ is Lipschitz continuous with parameter $L$ (Assumption~\ref{ass:lipschitz}). Then, for every iterate $x\in\mathbb{R}$ satisfying $\left|x-x^{\ast}\right|<\sqrt{\frac{\mu}{L}}r^{\ast}$, where $r^{\ast}\approx1.323$, the next iterate $x^{+}$ of Algorithm~\ref{alg:dth-order-newton-method}, where $d=3$, in the unconstrained setting (see equation~\eqref{eq:third-order-newton-univariate}) satisfies $$\text{$\left|x^{+}-x^{\ast}\right|<\frac{L}{\mu}\frac{1}{\left(r^{\ast}\right)^{2}}\left|x-x^{\ast}\right|^{3} \approx 0.572\cdot\frac{L}{\mu}{\left|x-x^{\ast}\right|^{3}}.$}$$
\end{proposition}

\subsection{Quasi-self-concordant objective functions}

\begin{definition}
    \label{def:qsc}
    A univariate, convex, three times continuously differentiable function $f:\mathbb{R}\to\mathbb{R}$ is called \emph{quasi-self-concordant} if there exists a constant $M>0$ such that $$\text{$\left|f^{\prime\prime\prime}\left(x\right)\right|\leq Mf^{\prime\prime}\left(x\right)$ for all $x\in\mathbb{R}$.}$$
\end{definition}

Let $f:\mathbb{R}\to\mathbb{R}$ be quasi-self-concordant with constant $M>0$ and strongly convex with constant $\mu>0$. Without loss of generality, we assume that $x^{\ast}=0$. We use the following shorthand notation: $$
\begin{aligned}
    g &:= \text{$f^{\prime}\left(x\right)$,} & g^{+} &:= \text{$f^{\prime}\left(x^{+}\right)$,} & g^{\ast} &:= \text{$f^{\prime}\left(x^{\ast}\right)$,}\\
    h &:= \text{$f^{\prime\prime}\left(x\right)$,} & h^{+} &:= \text{$f^{\prime\prime}\left(x^{+}\right)$,} & h^{\ast} &:= \text{$f^{\prime\prime}\left(x^{\ast}\right)$.}
\end{aligned}
$$ Since $x^{\ast}=0$ is the minimizer of $f$, we have $g^{\ast}=0$. By strong convexity, $h\geq\mu$, $h^{+}\geq\mu$, and $h^{\ast}\geq\mu$.

\begin{lemma}[Corollary 4 in \cite{rubbens2025performance}]
Let $\left\{\left(x_{k},g_{k},h_{k}\right)\right\}_{k=1}^{N}$ be a finite sequence. Then, there exists a function $f:\mathbb{R}\to\mathbb{R}$ that is quasi-self-concordant with cosntant $M>0$ satisfying $$\text{$g_{k}=f^{\prime}\left(x_{k}\right)$ and $h_{k}=f^{\prime\prime}\left(x_{k}\right)$ for all $k\in\left\{1,\dots,N\right\}$}$$ if and only if $$\text{$g_{l}-g_{k}\geq\frac{h_{k}+h_{l}}{M}-\frac{2}{M}\sqrt{h_{k}h_{l}}e^{-\frac{M}{2}\left(x_{l}-x_{k}\right)}$ for all $k,l\in\left\{1,\dots,N\right\}$.}$$
\end{lemma}

Let $t:=f^{\prime\prime\prime}\left(x\right)$. Then, $\left|t\right|\leq Mh$.

Using the implicit characterization of step $x^{+}-x$ together with the interpolation conditions above, we obtain the following PEP, whose optimal value is an upper bound on the largest possible distance between the next iterate $x^{+}$ and the minimizer $x^{\ast}=0$ when the starting point $x$ satisfies $\left|x\right|\leq r$, where $r>0$.

\begin{equation}
    \label{eq:pep-qsc}
    \begin{aligned}
        \tau\left(M,\mu,r\right):=\max_{x,x^{+},x^{\ast},g,g^{+},g^{\ast},h,h^{+},h^{\ast},t} &&& \left|x^{+}\right|\\
        \operatorname{s.t.} &&& \text{$g^{+}-g\geq\frac{h+h^{+}}{M}-\frac{2}{M}\sqrt{hh^{+}}e^{-\frac{M}{2}\left(x^{+}-x\right)}$, etc.}\\
         &&& gh+h^{2}\left(x^{+}-x\right)+\frac{ht}{2}\left(x^{+}-x\right)^{2}+\frac{t^{2}}{12}\left(x^{+}-x\right)^{3}=0\\
         &&& \text{$x^{\ast}=0$, $g^{\ast}=0$, $h\geq\mu$, $h^{+}\geq\mu$, $h^{\ast}\geq\mu$, $\left|t\right|\leq Mh$}\\
         &&& \text{$\left|x\right|\leq r$,}
    \end{aligned}
\end{equation} where $M,\mu,r>0$ are fixed. Note that taking $t:=0$ is always feasible and reduces the update to a Newton step. Therefore, using this PEP, we cannot prove any convergence properties that improve upon those of Newton's method. However, our aim  is to show that we can still prove local superlinear convergence to $x^{\ast}$. Indeed, Doikov \cite[\S 4]{doikov2025minimizing} showed (see also \S 5.5.2 in \cite{rubbens2025performance}) that Newton's method converges locally quadratically to $x^{\ast}$ for quasi-self-concordant $f$.

\medskip

\noindent The following scaling properties reduce the number of parameters. For every $\lambda>0$,
\begin{itemize}
    \item $x,x^{+},x^{\ast},g,g^{+},g^{\ast},h,h^{+},h^{\ast},t$ are feasible for $M,\mu,r$ if and only if $x,x^{+},x^{\ast},\lambda g,\lambda g^{+},\lambda g^{\ast},\lambda h,\lambda h^{+},\lambda h^{\ast},\lambda t$ are feasible for $M,\lambda\mu,r$, so $$\text{$\tau\left(M,\lambda\mu,r\right)=\tau\left(M,\mu,r\right)$,}$$
    \item $x,x^{+},x^{\ast},g,g^{+},g^{\ast},h,h^{+},h^{\ast},t$ are feasible for $M,\mu,r$ if and only if $\lambda x,\lambda x^{+},\lambda x^{\ast},\frac{g}{\lambda},\frac{g^{+}}{\lambda},\frac{g^{\ast}}{\lambda},\frac{h}{\lambda^{2}},\frac{h^{+}}{\lambda^{2}},\frac{h^{\ast}}{\lambda^{2}},\frac{t}{\lambda^{3}}$ are feasible for $\left(\frac{\mu}{\lambda^{2}},\frac{M}{\lambda},\lambda r\right)$, so $$\text{$\tau\left(\frac{M}{\lambda},\frac{\mu}{\lambda^{2}},\lambda r\right)=\lambda\tau\left(M,\mu,r\right)$.}$$
\end{itemize}
Combining these two scaling properties and using $\lambda = M$ gives $$\text{$\tau\left(M,\mu,r\right)=\frac{1}{M}\tau\left(1,\frac{\mu}{M^{2}},Mr\right)=\frac{1}{M}\tau\left(1,1,Mr\right)$.}$$

To study the worst-case performance for the class of univariate $f$ that are quasi-self-concordant with constant $M$ and strongly convex with constant $\mu$, we fix $M=\mu=1$ and solve problem~\eqref{eq:pep-qsc} numerically for a range of $r$. For each $r\in\left\{0.1,\ldots,2.0\right\}$,~we compute an upper bound on the optimal value using BARON. For each $r\in\left\{0.001,\ldots,2.000\right\}$, we compute a locally optimal value using SLSQP. The results are shown in Figures~\ref{fig:qsc-tau} and~\ref{fig:qsc-rate}.

\begin{figure}[H]
\centering
\includegraphics[width=0.6\linewidth]{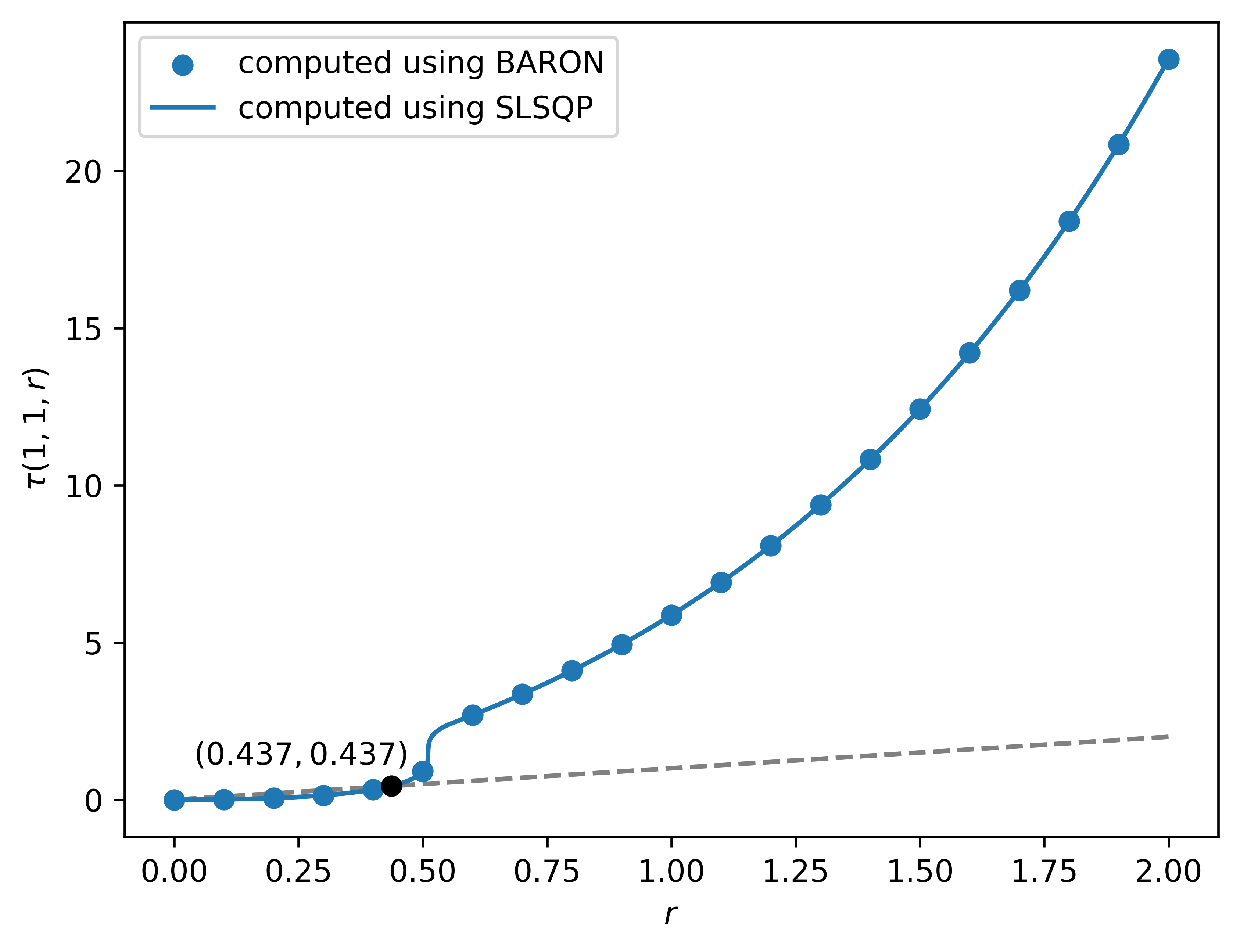}
\caption{Worst-case distance $\tau\left(1,1,r\right)$ as a function of the starting radius $r$.}
\label{fig:qsc-tau}
\end{figure}

\begin{figure}[H]
\centering
\includegraphics[width=0.6\linewidth]{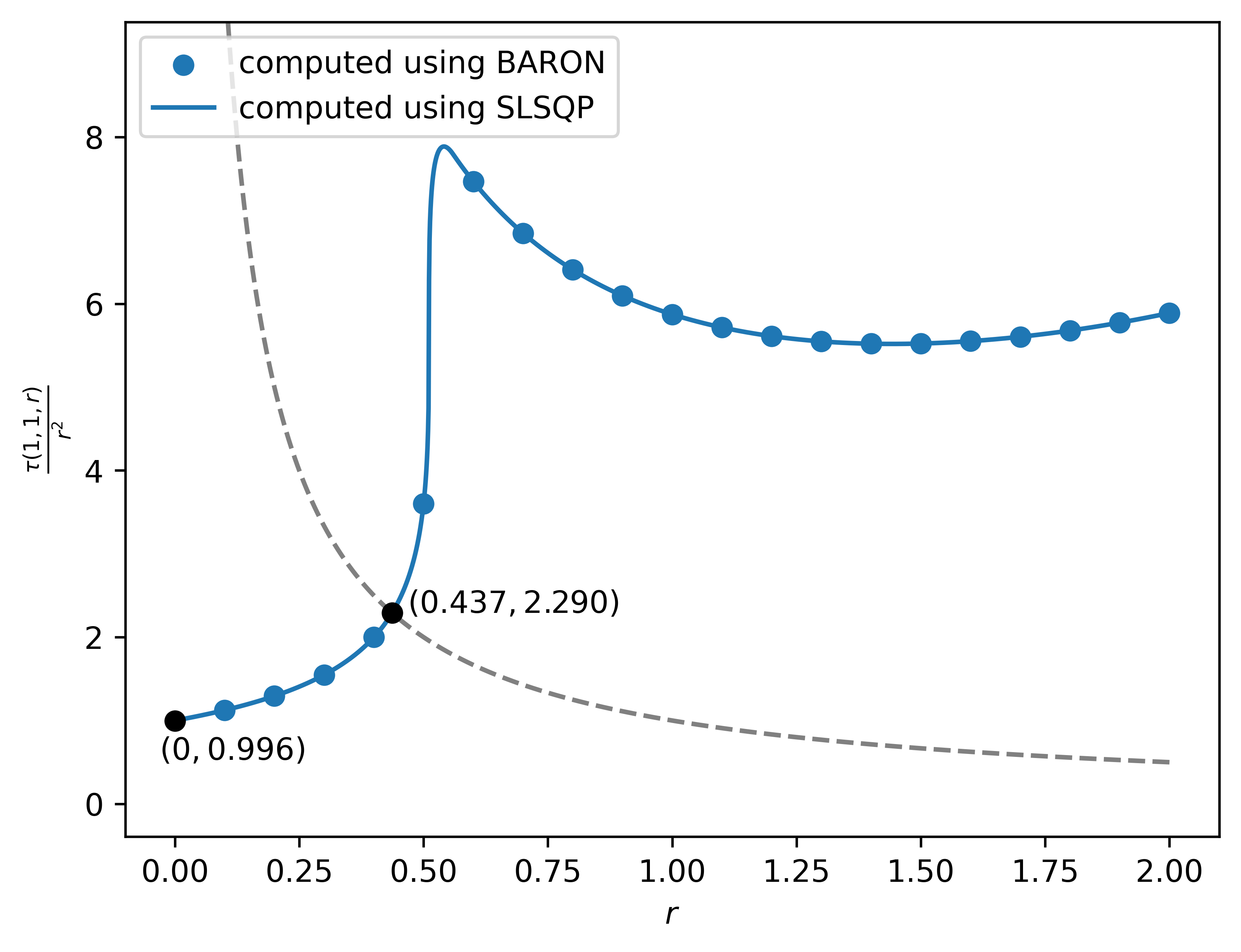}
\caption{Quadratically normalized worst-case distance $\frac{\tau\left(1,1,r\right)}{r^{2}}$ as a function of the starting radius $r$.}
\label{fig:qsc-rate}
\end{figure}

Figure~\ref{fig:qsc-tau} shows that the worst-case distance $\tau\left(1,1,r\right)$ grows slowly for small values of $r$, before increasing rapidly once $r\approx0.5$. The intersection with the dashed line gives the largest starting radius for which the next iterate remains inside this radius in the worst case. Let $r^{\ast}>0$ denote the unique solution to $\tau\left(1,1,r^{\ast}\right)=r^{\ast}$, which is computed numerically to be $r^{\ast}\approx0.437$. Using the scaling properties, $$\text{$\tau\left(M,\mu,r\right)=\frac{1}{M}\tau\left(1,1,Mr\right)<r\iff\tau\left(1,1,Mr\right)<Mr\iff Mr<r^{\ast}\iff r<\frac{1}{M}r^{\ast}$.}$$

Figure~\ref{fig:qsc-rate} shows that $\frac{\tau\left(1,1,r\right)}{r^{2}}$ converges to a finite limit as $r\to0$, corresponding to the asymptotic qua-dratic convergence rate. Therefore, the method converges at least quadratically for every univariate $f$ that is quasi-self-concordant with constant $M$ and strongly convex with constant $\mu$ whenever $$\text{$r<\frac{1}{M}r^{\ast}$, where $r^{\ast}\approx0.437$.}$$ The worst-case quadratic convergence rate satisfies $$\text{$\frac{\tau\left(M,\mu,r\right)}{r^{2}}=M\cdot\frac{\tau\left(1,1,Mr\right)}{\left(Mr\right)^{2}}<M\cdot\frac{\tau\left(1,1,r^{\ast}\right)}{\left(r^{\ast}\right)^{2}}=M\frac{1}{r^{\ast}}$, where $\frac{1}{r^{\ast}}\approx2.290$.}$$

In summary, we obtain the following bound on the local quadratic rate of convergence. (As discussed, this bound is not tight, since the PEP is a relaxation here.)
\begin{proposition}
    Let $f:\mathbb{R}\to\mathbb{R}$ be univariate. Assume that $f$ is quasi-self-concordant with parameter $M$ (see Definition~\ref{def:qsc}) and that $f$ is strongly convex with parameter $\mu$ (Assumption~\ref{ass:strongly-convex}). Then, for every iterate $x\in\mathbb{R}$ satisfying $\left|x-x^{\ast}\right|<\frac{1}{M}r^{\ast}$, where $r^{\ast}\approx0.437$, the next iterate $x^{+}$ of Algorithm~\ref{alg:dth-order-newton-method}, where $d=3$, in the unconstrained setting (see equation~\eqref{eq:third-order-newton-univariate}) satisfies $$\text{$\left|x^{+}-x^{\ast}\right|<M\frac{1}{r^{\ast}}\left|x-x^{\ast}\right|^{2} \approx 2.29\cdot M\left|x-x^{\ast}\right|^{2}$.}$$
\end{proposition}

\section{Global convergence in the setting of SOS-convex polynomial constraints}

Since the tensor of $d$\textsuperscript{th}-order partial derivatives of $f$ is assumed to be Lipschitz continuous (Assumption~\ref{ass:lipschitz}), it follows from Theorem~\ref{thm:baes} that $$\text{$f\left(y\right)\leq T_{x,d}\left(y\right)+\frac{L}{\left(d+1\right)!}\left\|y-x\right\|^{d+1}$.}$$ In this section, we assume that $d$ is odd, so that the right-hand side is a polynomial. Furthermore, we assume that an upper bound $M$ on the Lipschitz constant $L$ is known. We redefine the regularized $d$\textsuperscript{th}-order Taylor expansion $\psi_{x,d}:\mathbb{R}^{n}\to\mathbb{R}$ of $f$ at $x$ by $$\text{$\psi_{x,d}\left(y\right):=T_{x,d}\left(y\right)+\max\left\{\frac{M}{\left(d+1\right)!},t\left(x\right)\right\}\left\|y-x\right\|^{d+1}$,}$$ so that $f\left(y\right)\leq\psi_{x,d}\left(y\right)$ for all $y\in\mathbb{R}^{n}$, and define the next iterate $x^{+}$ as the optimal solution to the problem
\begin{equation}
    \label{eq:global_subproblem}
    \begin{aligned}
        \min_{y\in\mathbb{R}^{n}} &&& \psi_{x,d}\left(y\right) && \\
        \operatorname{s.t.} &&& \text{$g_{i}\left(y\right)\leq0$,} && \text{$i\in\left\{1,\dots,m\right\}$.}
    \end{aligned}
\end{equation}

\begin{algorithm}[H]
    \caption{Globally convergent $d$\textsuperscript{th}-order Newton method ($d$ odd) (cf., Algorithm~2 in \cite{ahmadi2024higher})}
    \label{alg:globally-convergent-dth-order-newton-method}
    \KwIn{Starting point $x\in\mathbb{R}^{n}$}
    Solve problem~\eqref{eq:t} to compute $t\left(x\right)$\;
    Let $x^{+}$ be the optimal solution to problem~\eqref{eq:global_subproblem}\;
    Set $x\gets x^{+}$ and repeat\;
\end{algorithm}

\medskip

Algorithm~\ref{alg:globally-convergent-dth-order-newton-method} is called a \emph{majorization--minimization (MM) algorithm}. Given an iterate $x\in\mathbb{R}^{n}$, the next iterate $x^{+}$ is defined as the minimizer of a function that majorizes $f$ subject to the constraints. Under some assumptions, an MM algorithm converges globally to $x^{\ast}$ (see, e.g., \cite{lange2016mm}). We give a proof tailored to our setting.

\newpage

\begin{theorem}
    \label{thm:global-convergence}
    Assume that $f$ is strongly convex (Assumption~\ref{ass:strongly-convex}) and that the tensor of $d$\textsuperscript{th}-order partial derivatives of $f$ is Lipschitz continuous (Assumption~\ref{ass:lipschitz}). Then,
    \begin{itemize}
        \item for any starting point $x_{0}\in\mathbb{R}^{n}$, the sequence of iterates $\left\{x_{k}\right\}_{k=1}^{\infty}$ of Algorithm~\ref{alg:globally-convergent-dth-order-newton-method} converges to $x^{\ast}$,
        \item for every iterate $x\in\mathbb{R}^{n}$ satisfying $\left\|x-x^{\ast}\right\|\leq r$, where $r:=\left(\frac{\left(d-1\right)!\mu}{2L}\right)^{\frac{1}{d-1}}$, the next iterate $x^{+}$ of Algorithm~\ref{alg:globally-convergent-dth-order-newton-method} satisfies $$\text{$\left\|x^{+}-x^{\ast}\right\|\leq\left(\frac{d+1}{2}\right)^{2}\frac{2}{\mu}\left(\frac{L}{d!}+\max\left\{\frac{M}{\left(d+1\right)!},\sup_{\left\|z-x^{\ast}\right\|\leq r}t\left(z\right)\right\}\left(d+1\right)\right)\left\|x-x^{\ast}\right\|^{d}$.}$$
    \end{itemize}
\end{theorem}

\begin{proof}
    Since $x_{k+1}$ minimizes $\psi_{x_{k},d}$ over $C$, $$f\left(x_{k+1}\right)\leq\psi_{x_{k},d}\left(x_{k+1}\right)\leq\psi_{x_{k},d}\left(x_{k}\right)=f\left(x_{k}\right).$$ Hence, the sequence of function values $\left\{f\left(x_{k}\right)\right\}_{k=0}^{\infty}$ is decreasing. Since it is bounded below by $f\left(x^{\ast}\right)$, it is convergent. By the fundamental theorem of calculus, $$\text{$f\left(x_{k}\right)=f\left(x^{\ast}\right)+\left\langle\nabla f\left(x^{\ast}\right),x_{k}-x^{\ast}\right\rangle+\int_{0}^{1}\int_{0}^{t}\left\langle\nabla^{2}f\left(x^{\ast}+s\left(x_{k}-x^{\ast}\right)\right)\left(x_{k}-x^{\ast}\right),x_{k}-x^{\ast}\right\rangle ds\,dt$.}$$ Since $\nabla f\left(x^{\ast}\right)=0$ and $\nabla^{2}f\left(x\right)\succeq\mu I$ for all $x\in\mathbb{R}^{n}$, $$\text{$f\left(x_{k}\right)\geq f\left(x^{\ast}\right)+\frac{\mu}{2}\left\|x-x^{\ast}\right\|^{2}$.}$$ Therefore, $$\text{$\left\|x_{k}-x^{\ast}\right\|^{2}\leq\frac{2}{\mu}\left(f\left(x_{k}\right)-f\left(x^{\ast}\right)\right)\leq\frac{2}{\mu}\left(f\left(x_{0}\right)-f\left(x^{\ast}\right)\right)$.}$$ Hence, the sequence of iterates $\left\{x_{k}\right\}_{k=1}^{\infty}$ is bounded. By the Bolzano--Weierstrass theorem, there~exists~a~sub- sequence of iterates $\left\{x_{k_{j}}\right\}_{j=1}^{\infty}$ that converges to a point $\bar{x}\in C$. Since $x_{k_{j}+1}$ minimizes $\psi_{x_{k_{j}},d}$ over $C$, $$\text{$\psi_{x_{k_{j+1}},d}\left(x_{k_{j+1}}\right)=f\left(x_{k_{j+1}}\right)\leq f\left(x_{k_{j}+1}\right)\leq\psi_{x_{k_{j}},d}\left(x_{k_{j}+1}\right)\leq\psi_{x_{k_{j}},d}\left(y\right)$ for all $y\in C$.}$$ Taking the limit as $j\to\infty$ gives $\psi_{\bar{x},d}\left(\bar{x}\right)\leq\psi_{\bar{x},d}\left(y\right)$ for all $y\in C$. Hence, $\bar{x}$ minimizes $\psi_{\bar{x},d}$ over $C$. Thus, $$\text{$\left\langle\nabla f\left(\bar{x}\right),y-\bar{x}\right\rangle=\left\langle\nabla\psi_{\bar{x},d}\left(\bar{x}\right),y-\bar{x}\right\rangle\geq0$ for all $y\in C$.}$$ Thus, $\bar{x}$ minimizes $f$ over $C$. Hence, $\lim_{k\to\infty}f\left(x_{k}\right)=\lim_{j\to\infty}f\left(x_{k_{j}}\right)=f\left(\bar{x}\right)=f\left(x^{\ast}\right)$. Therefore, $$\text{$\left\|x_{k}-x^{\ast}\right\|^{2}\leq\frac{2}{\mu}\left(f\left(x_{k}\right)-f\left(x^{\ast}\right)\right)\to0$ as $k\to\infty$.}$$ Hence, $x_{k}\to x^{\ast}$ as $k\to\infty$.

    The second result follows from arguments analogous to those in the proof of Theorem~\ref{thm:convergence}.
\end{proof}

It is natural to ask whether there exists a function $f$ satisfying $$\text{$t\left(x\right)\geq\frac{L}{\left(d+1\right)!}$ for all $x\in\mathbb{R}^{n}$,}$$ so that Algorithm~\ref{alg:dth-order-newton-method} is equivalent to Algorithm~\ref{alg:globally-convergent-dth-order-newton-method} with $M:=L$, and therefore converges globally to $x^{\ast}$. If $\left(n,d\right)=\left(1,3\right)$ and $L\neq0$, the answer is unfortunately no, as the following result shows.

\begin{proposition}
    There exists no univariate, three times continuously differentiable function $f:\mathbb{R}\to\mathbb{R}$ that is strongly convex (Assumption~\ref{ass:strongly-convex}) and whose third derivative is Lipschitz continuous (Assumption~\ref{ass:lipschitz}), where $L\neq0$, satisfying $t\left(x\right)\geq\frac{L}{24}$ for all $x\in\mathbb{R}$.
\end{proposition}

\begin{proof}
    Suppose that such a function $f:\mathbb{R}\to\mathbb{R}$ does exist. Ahmadi, Chaudhry, and Zhang \cite{ahmadi2024higher} showed that $$\text{$t\left(x\right)=\frac{\left(f^{\prime\prime\prime}\left(x\right)\right)^{2}}{48f^{\prime\prime}\left(x\right)}$.}$$ Therefore, $\left(f^{\prime\prime\prime}\left(x\right)\right)^{2}\geq48f^{\prime\prime}\left(x\right)\cdot\frac{L}{24}\geq2\mu L>0$ for all $x\in\mathbb{R}$. Since the third derivative of $f$ is continuous, either $f^{\prime\prime\prime}\left(x\right)\geq\sqrt{2\mu L}$ for all $x\in\mathbb{R}$ or $f^{\prime\prime\prime}\left(x\right)\leq-\sqrt{2\mu L}$ for all $x\in\mathbb{R}$.
    
    \smallskip
    
    \begin{itemize}
        \item In the former case, $f^{\prime\prime}\left(x\right)\leq f^{\prime\prime}\left(0\right)+\sqrt{8\mu L}x\to-\infty$ as $x\to-\infty$.
        \item In the latter case, $f^{\prime\prime}\left(x\right)\leq f^{\prime\prime}\left(0\right)-\sqrt{8\mu L}x\to-\infty$ as $x\to\infty$.
    \end{itemize}
    
    \smallskip
    
    \noindent This contradicts the fact that $f$ is strongly convex.
\end{proof}

\section{Generalized settings}

\subsection{Local optimization}

Previously, we assumed that $f$ was globally strongly convex (Assumption~\ref{ass:strongly-convex}) and that the tensor of $d$\textsuperscript{th}-order partial derivatives of $f$ was globally Lipschitz continuous (Assumption~\ref{ass:lipschitz}). Under these assumptions, Algorithm~\ref{alg:dth-order-newton-method} converged locally to the globally optimal solution $x^{\ast}$ with order $d$ (Theorem~\ref{thm:convergence}). We now claim that Algorithm~\ref{alg:dth-order-newton-method} can also be used for local optimization.
Let $x^{\ast}$ be a locally optimal solution. We assume that $f$ satisfies the following regularity conditions.

\begin{assumption}
    \label{ass:locally-strongly-convex}
    $f$ is strongly convex on a neighborhood of $x^{\ast}$. That is, there exist a constant $\mu>0$ and a radius $r_{\mu}>0$ such that $$\text{$\nabla^{2}f\left(x\right)\succeq\mu I$ for all $x\in B\left[x^{\ast},r_{\mu}\right]$.}$$
\end{assumption}

\begin{assumption}
    \label{ass:locally-lipschitz}
    The tensor of $d$\textsuperscript{th}-order partial derivatives of $f$ is Lipschitz continuous on a neighborhood of $x^{\ast}$. That is, there exist a constant $L\geq0$ and a radius $r_{L}>0$ such that $$\text{$\left\|\nabla^{d}f\left(y\right)-\nabla^{d}f\left(x\right)\right\|\leq L\left\|y-x\right\|$ for all $x,y\in B\left[x^{\ast},r_{L}\right]$.}$$
\end{assumption}

Note that Assumption~\ref{ass:locally-strongly-convex} is satisfied if and only if $\nabla^{2}f\left(x^{\ast}\right)\succ0$. A sufficient condition for Assumption~\ref{ass:locally-lipschitz} is that $f$ is $d+1$ times continuously differentiable on a neighborhood of $x^{\ast}$.

\medskip

By arguments analogous to those in the proofs of Theorems~\ref{thm:well-definedness} and~\ref{thm:convergence}, Algorithm~\ref{alg:dth-order-newton-method} is well defined on a neighborhood of $x^{\ast}$ in the sense that, for every iterate $x\in\mathbb{R}^{n}$,
    \begin{itemize}
        \item problem~\eqref{eq:t} is feasible,
        \item problem~\eqref{eq:subproblem} has a unique optimal solution,
    \end{itemize}
and, for every iterate $x\in\mathbb{R}^{n}$ satisfying $\left\|x-x^{\ast}\right\|\leq r$, where $r:=\min\left\{r_{\mu},r_{L},\left(\frac{\left(d-1\right)!\mu}{2L}\right)^{\frac{1}{d-1}}\right\}$, the next~iterate $x^{+}$ of Algorithm~\ref{alg:dth-order-newton-method} satisfies $$\text{$\left\|x^{+}-x^{\ast}\right\|\leq\left(\frac{d^{\prime}}{2}\right)^{2}\frac{2}{\mu}\left(\frac{L}{d!}+\left(\max_{\left\|z-x^{\ast}\right\|\leq r}t\left(z\right)\right)d^{\prime}\max\left\{r,1\right\}\right)\left\|x-x^{\ast}\right\|^{d}$.}$$

\subsection{General convex constraints}

Note that, in the proofs of Theorems~\ref{thm:well-definedness} and~\ref{thm:convergence}, we do not use the fact that the feasible set is described by SOS-convex polynomial, but only that it is closed and convex. Therefore, Algorithm~\ref{eq:main-problem} can be used to solve any optimization problem of the form $$
\begin{aligned}
    \min_{x\in\mathbb{R}^{n}} &&& f\left(x\right)\\
    \operatorname{s.t.} &&& x\in C,
\end{aligned}
$$ where $C\subseteq\mathbb{R}^{n}$ is closed and convex, provided that minimizing an SOS-convex polynomial over $C$ is tractable. A few examples are:
\begin{itemize}
    \item If $C$ is described by SOS-convex polynomial inequalities and linear matrix inequalities, then, by arguments analogous to those in the proof of Theorem~3.3 in \cite{lasserre2009convexity}, minimizing an SOS-convex polynomial over $C$ can be reduced in time polynomial in $n$ to an SDP.
    \item If a point $x_{0}$ and radii $r,R>0$ such that $B\left(x_{0},r\right)\subseteq C\subseteq B\left(x_{0},R\right)$ are known, and the separation problem for $C$ can be solved polynomial time, then the problem of minimizing an SOS-convex polynomial over $C$ can be solved to arbitrary accuracy in polynomial time using the ellipsoid method (see, e.g., Chapter~3 in \cite{grotschel1981ellipsoid}).
\end{itemize}

\section*{Acknowledgement}

This research was funded by the Dutch Research Council (NWO) for the project \emph{The Synergy between Semi- definite Programming and Approximation Theory} (OCENW.M.23.050).

\end{document}